\documentclass[final]{siamltex}

\usepackage{epsfig}
\usepackage{fancyhdr}
\usepackage{amsmath,amssymb}
\usepackage {graphicx}
\usepackage[ruled,linesnumbered]{algorithm2e}
\usepackage{enumitem}
\usepackage{mathrsfs}
\usepackage{booktabs}
\usepackage{multirow}
\usepackage{float}
\usepackage{appendix}

\newtheorem{remark}{Remark}[section]

\title{Sparse power methods for large and sparse higher-order PageRank problem}

\author{Jun Huang\thanks{School of Mathematics, China University of Mining and Technology, 221116, Jiangsu, P.R. China. E-mail: {\tt hj\_math@126.com}.}
       \and Gang Wu\thanks{Corresponding author. School of Mathematics,
China University of Mining and Technology, Xuzhou, 221116, Jiangsu, P.R. China.
E-mail: {\tt gangwu@cumt.edu.cn} and {\tt gangwu76@126.com}. This work is supported by the Fundamental Research Funds for the Central Universities of China under grant 2019XKQYMS89.}
}

\begin{document}

\maketitle

\begin{abstract}
Higher-order Markov chain plays an important role in analyzing a variety of stochastic processes over time
in multi-dimensional spaces. One of the important applications is the study of the multilinear PageRank problem. However, unlike the high-order PageRank problem, the existence and uniqueness of multilinear PageRank vector is not guaranteed theoretically. Moreover, as an alternative, the approximation of the multilinear PageRank vector to the high-order PageRank tensor can be poor in practice. A commonly used technique for the higher-order PageRank problem is the power method that is computationally intractable for large-scale problems. The truncated power method proposed recently provides us with another idea to solve this problem, however, its accuracy and efficiency can be poor in practical computations. In this work, we revisit the higher-order PageRank problem and consider how to solve it efficiently. The contribution of this work is as follows.
First, we accelerate the truncated power method for high-order PageRank. In the improved version, it is neither to form and store the vectors arising from the dangling states, nor to store an auxiliary matrix.
Second, we propose a truncated power method with {\it partial updating} to further release the overhead, in which
one only needs to update some important columns of the approximation in each iteration.
On the other hand, the truncated power method solves a modified high-order PageRank model for convenience, which is not mathematically equivalent to the original one.
Thus, the third contribution of this work is to propose a sparse power method with partial updating for the original higher-order PageRank problem. The convergence of all the proposed methods are discussed.
Numerical experiments on large and sparse real-world and synthetic data sets are performed. The numerical results show the superiority of our new algorithms over some state-of-the-art ones for large and sparse higher-order PageRank problems.
\end{abstract}



\section{Introdution}
Markov chain is a stochastic process that takes advantage of past and current prior information to predict the future information \cite{CKL1979A, 2021Stationary, WJ2009A}.
The first-order Markov chain only uses current information, and one of its representative applications is Google's PageRank \cite{Gleich2105PBW,langville2011google,page1999pagerank,Wu2012A}.
However, the first-order Markov chain does not consider historical information that is often significant in practical use.
Therefore, the time limitation makes this model sometimes seriously affect the accuracy of prediction, thereby reducing the credibility of conclusions.

Higher-order Markov chain is a generalization of the first-order Markov chain \cite{Ching2004,Ching2006,Ching2008,Raftery1985A}. It exploits the current state of several consecutive moments to predict the future state.
The advantage of the higher-order Markov chain lies in its high memory and ability to consider huge historical information, allowing us to make reasonable judgments in analyzing a variety of stochastic processes over time in multi-dimensional spaces \cite{Wen2020Multilinear}.
Corresponding to the PageRank model, the higher-order Markov chain also have an important application named higher-order PageRank \cite{Gleich2014Multilinear}.
However, one has to store a {\it dense} higher-order PageRank tensor in the higher-order PageRank problem, which is prohibitive for large-scale data sets.

In order to partially overcome this difficulty, Gleich, Lim, and Yu considered multilinear PageRank instead \cite{Gleich2014Multilinear,Li2014A}, which is an alternative to higher-order PageRank.
The multilinear PageRank problem has received great attention in recent years; see \cite{Cipolla2019Extrapolation,Gleich2014Multilinear,Pei2020A,Guo2018A,HW,Li2017The, Wen2020Multilinear,Li2014A,Dongdong2019Relaxation,Meini2018Perron} and the references therein.
For instance, Gleich {\it et al.} proposed five algorithms for the multilinear PageRank problem \cite{Gleich2014Multilinear}.
Cipolla {\it et al.} took advantage of extrapolation skills to accelerate the simplified topological $\varepsilon$-algorithm in its restarted form for multilinear PageRank \cite{Cipolla2019Extrapolation}.
Guo {\it et al.} proposed a modified Newton method and analyzed a residual-based error bound for the multilinear PageRank vector \cite{Pei2020A, Guo2018A}.
Meini and Poloni investigated Perron-based algorithms for multilinear PageRank \cite{Meini2018Perron}.
Liu {\it et al.} considered several relaxation algorithms for solving the tensor equation arising from the multilinear PageRank problem \cite{Dongdong2019Relaxation}.

Unfortunately, for a given damping factor $0<\alpha<1$, the existence and uniqueness of multilinear PageRank vector is not
guaranteed theoretically \cite{Gleich2014Multilinear}. Recently, some theoretical results are established on the existence and uniqueness of the multilinear PageRank vector; see \cite{Li2017The, Wen2020Multilinear,Li2014A,Dongdong2019Relaxation} and the references therein.
However, the conditions for these results are often complicated and some of them are difficult to use in practice.
As a comparison, the higher-order PageRank tensor always exists and is unique \cite{Gleich2014Multilinear}.
On the other hand, the approximation of the multilinear PageRank vector to the high-order PageRank tensor can be poor.
For instance, for a multilinear PageRank problem of order 3, the approximation obtained from multilinear PageRank is a rank-$1$ {\it symmetric} matrix, and it is a poor approximation to a two-order PageRank tensor that is a {\it nonsymmetric} matrix.

Thus, it is interesting to solve the higher-order PageRank problem directly, provided that both computational overhead and storage requirements could be reduced significantly.
In \cite{WuYubao}, Wu {\it et al.} proposed to formulate second-order Markov chains by using edge-to-edge transition
probabilities. The main idea is to construct the corresponding edge-to-edge formulated matrix according to the sparsity structure and dangling nodes structure. As the number of edges is about $n^2$, where $n$ is the number of
nodes, the computational cost will be very high for large-scale problems.
Recently, Ding {\it et al.} considered the large and sparse higher-order PageRank problem \cite{Ding2018Fast}.
The main idea of their work is to approximate large-scale solutions of sparse Markov chains by
two components called sparsity component and rank-one component, and
a truncated power method was proposed. In this method, one can determine solutions by formulating and solving sparse and rank-one
optimization problems via closed form solutions.
However, as the computational model is a modification to the original higher-order PageRank model, the error between the approximation and the exact solution can be large. Moreover, the truncated power method may converge slowly for large-scale problems.
Therefore, it is urgent to seek new technologies to accelerate the truncated power method.

Instead of the multilinear PageRank problem, we revisit the higher-order PageRank problem and consider how to solve it efficiently in this work. The contribution of this work is as follows.
First, we accelerate the truncated power method for high-order PageRank. In the improved version, there is no need to form and store the vectors arising from the dangling states, nor to store an auxiliary matrix of size $n$, where $n$ is the number of nodes.
Second, motivated by the fact that we are often concerned with several webpages with {\it top} PageRank values,
we propose a truncated power method with {\it partial updating} to further release the overhead. In this strategy,
one only needs to update some {\it important} columns of the approximation in each iteration.
Third, the truncated power method applies to a modified high-order PageRank model that is not mathematically equivalent to the original one.
In this work, we propose a sparse power method with partial updating for the original higher-order PageRank problem.

The organization of this paper is as follows.
In Section \ref{back}, we briefly review the higher-order PageRank problem and the multilinear PageRank problem.
In Section \ref{TPM}, we introduce the truncated power method due to Ding {\it et al.} for the higher-order PageRank problem.
To improve the truncated power method, we propose a variation on the truncated power method and a truncated power method with partial updating in Section \ref{rt}. The convergence of the proposed methods is established.
In Section \ref{irt}, we put forward a sparse power method with higher accuracy for the original higher-order PageRank problem, and take advantage of the idea of partial updating to accelerate this method. The convergence of the two methods is analysed.
Numerical experiments are performed on large and sparse real-world and synthetic data sets in Section \ref{ne}, to show the efficiency of the proposed strategy and the superiority of the proposed algorithms. Some concluding remarks are given in Section \ref{sec7}.


\section{Higher-order PageRank and multilinear PageRank}\label{back}
An order-$m$, $n$-dimensional tensor has $m$ indices that range from $1$ to $n$.
Without loss of generality, we assume that $m=3$ throughout this paper.
A transition tensor is a stochastic tensor that is nonnegative and where the sum over the first index is $1$.
In this section, we briefly review the higher-order PageRank and the multilinear PageRank problems \cite{Gleich2014Multilinear}.
In \cite{Gleich2014Multilinear}, Gleich {\it et al.} introduced the following {\it higher-order PageRank problem} that is a generalization to the classical PageRank problem \cite{Gleich2105PBW,langville2011google,page1999pagerank,Wu2012A}.
\begin{definition}\cite[Definition 3.2]{Gleich2014Multilinear}
Let $\mathcal{P}$ be an order-$m$, $n$-dimensional transition tensor representing an order-$(m-1)$ Markov chain, $\alpha \in (0,1)$ be a probability, and $\mathbf{v} \in \mathbb{R}^{n}$ be a stochastic vector.
Then the higher-order PageRank tensor $\mathcal{X}$ is the order-$(m-1)$, $n$-dimensional tensor that solves the linear system
\begin{equation}\label{2.2}
\mathcal{X}(i, j, \ldots, \ell)=\alpha \sum_{k} \mathcal{P}(i, j, \ldots, \ell, k) \mathcal{X}(j, \ldots, \ell, k)+(1-\alpha) v_{i} \sum_{k} \mathcal{X}(j, \ldots, \ell, k).
\end{equation}
\end{definition}
That is, the tensor $\mathcal{X}$ is the stationary distribution of a higher-order Markov chain, and we have the following result.
\begin{theorem}\cite[Corollary 3.5]{Gleich2014Multilinear}\label{lem-2}
The higher-order PageRank stationary distribution tensor $\mathcal{X}$ always exists and is unique. Also, the standard PageRank iteration will result in a $1$-norm error of $2(1 - (1 - \alpha)^{m-1})^{k}$ after $(m-1)k$ iterations.
\end{theorem}

Specifically, as $m=3$, we have $\mathcal{X}=\boldsymbol{X} \in \mathbb{R}^{n \times n}$, and \eqref{2.2} reduces to
\begin{equation}\label{3.1}
\boldsymbol{X_{ij}} = \alpha \sum_{k} \mathcal{P}_{ijk} \boldsymbol{X_{jk}} + (1-\alpha) v_{i} \sum_{k} \boldsymbol{X_{jk}}.
\end{equation}
The above equation can be reformulated as \cite{Gleich2014Multilinear}
\begin{equation}\label{3.4}
vec(\boldsymbol{X}) = [\alpha \boldsymbol{P} + (1-\alpha)\boldsymbol{V}] vec(\boldsymbol{X}),
\end{equation}
where $\boldsymbol{P} \in \mathbb{R}^{n^{2} \times n^{2}}$, $\boldsymbol{I} \in \mathbb{R}^{n \times n}$ is an identity matrix, and
\begin{equation}
\boldsymbol{V} = \mathbf{e}^{\boldsymbol{T}} \otimes \boldsymbol{I} \otimes \mathbf{v} \in \mathbb{R}^{n^{2} \times n^{2}}.
\end{equation}
Here $\mathbf{e}$ is the vector of all ones, $vec(\cdot)$ is the ``vec operator" that stacks the columns
of a matrix into a long vector, and $\otimes$ is the Kronecker product.

Therefore,
\begin{equation}\label{3.2}
\boldsymbol{X}(:,j)=\alpha \mathbf{P}_{j} \boldsymbol{X}(j,:)^{\boldsymbol{T}}+(1-\alpha) \widetilde{\mathbf{V}}\boldsymbol{X}(j,:)^{\boldsymbol{T}},\quad j=1,2,\ldots,n,
\end{equation}
where $\widetilde{\mathbf{V}} = \mathbf{v} \mathbf{e}^{\boldsymbol{T}} \in \mathbb{R}^{n \times n}$, and
\begin{equation}\label{eq25}
\mathbf{P}_{j}=\left[\begin{array}{cccc}
\mathcal{P}_{1 j 1} & \mathcal{P}_{1 j 2} & \cdots & \mathcal{P}_{1 j n} \\
\mathcal{P}_{2 j 1} & \mathcal{P}_{2 j 2} & \cdots & \mathcal{P}_{2 j n} \\
\vdots & \vdots & \ddots & \vdots \\
\mathcal{P}_{n j 1} & \mathcal{P}_{n j 2} & \cdots & \mathcal{P}_{n j n}
\end{array}\right] \in \mathbb{R}^{n \times n},\quad j=1,2,\ldots,n,
\end{equation}
are column stochastic matrices.
In practice, however, we can only get the {\it sparse} matrices $\mathbf{Q}_{j}$ rather than the dense matrices $\mathbf{P}_{j}$ \cite{Gleich2014Multilinear}
\begin{equation*}
\mathbf{Q}_{j}=\left[\begin{array}{cccc}
\mathcal{Q}_{1 j 1} & \mathcal{Q}_{1 j 2} & \cdots & \mathcal{Q}_{1 j n} \\
\mathcal{Q}_{2 j 1} & \mathcal{Q}_{2 j 2} & \cdots & \mathcal{Q}_{2 j n} \\
\vdots & \vdots & \ddots & \vdots \\
\mathcal{Q}_{n j 1} & \mathcal{Q}_{n j 2} & \cdots & \mathcal{Q}_{n j n}
\end{array}\right] \in \mathbb{R}^{n \times n},\quad j=1,2,\ldots,n.
\end{equation*}
The matrices $\{\mathbf{Q}_{j}\}_{j=1}^n$ model the original sparse data, and are often not column random matrices. Let $\mathbf{d}_{j} \in \mathbb{R}^{n}$ be the vectors contain the entries arising from the dangling states of $\mathbf{Q}_{j}$\footnote{When the elements $\mathcal{Q}_{ijk}$ are equal to zero for all $i=1,2,\ldots,n$, the state $(j,k)$ is called the {\it dangling state}, which refers to no connection from the state $(j,k)$ \cite{Gleich2014Multilinear}.}.
Similar to the classical PageRank model \cite{langville2011google,page1999pagerank}, we can make dangling corrections to $\mathbf{Q}_{j}$, i.e.,
$$
\mathbf{P}_{j} = \mathbf{Q}_{j} + \frac{1}{n} \mathbf{e} \mathbf{d}_{j}^{\boldsymbol{T}},\quad j=1,2,\ldots,n,
$$
which are stochastic matrices.

It is challenging to compute the stationary probability distribution of large and sparse high-order Markov
chains.
Mathematically, as the higher-order PageRank problem is the stationary distribution of a high-order Markov chain, the power method is a nature choice for its simplicity and efficiency \cite{WJ2009A}.
According to \eqref{3.2}, given a nonnegative matrix $\mathbf{A} \in \mathbb{R}_{+}^{n \times n}$, we define the operator $\mathscr{W}:~\mathbb{R}^{n \times n} \mapsto \mathbb{R}^{n \times n}$ as follows
\begin{eqnarray}\label{2.6}
\mathscr{W}(\mathbf{A})(:,j)&=&\alpha (\mathbf{P}_{j}\mathbf{A}(j,:)^{\boldsymbol{T}}) + (1-\alpha)\widetilde{\mathbf{V}}\mathbf{A}(j,:)^{\boldsymbol{T}}\nonumber\\
&=&{\small \alpha (\mathbf{Q}_{j} \mathbf{A}(j,:)^{\boldsymbol{T}}) + \frac{\alpha}{n}(\mathbf{d}_{j}^{\boldsymbol{T}} \mathbf{A}(j,:)^{\boldsymbol{T}}) \mathbf{e} + (1-\alpha)\|\mathbf{A}(j,:)\|_{1}\mathbf{v}}.
\end{eqnarray}
Given the initial guess $\boldsymbol{X}^{(0)} \in \mathbb{R}_{+}^{n}$ with its $l_{1}$-norm being 1, i.e., $\|\boldsymbol{X}^{(0)}\|_{l_{1}} = 1$, then the power method resorts to the following iterations
\begin{equation*}
\boldsymbol{X}^{(q+1)} = \mathscr{W}(\boldsymbol{X}^{(q)}),\quad q=0,1,\ldots,
\end{equation*}
and the main overhead is to update the columns of approximation in the following way
\begin{equation}\label{4.3}
\boldsymbol{X}^{(q+1)}(:,j) = \alpha (\mathbf{Q}_{j} \boldsymbol{X}^{(q)}(j,:)^{\boldsymbol{T}}) + \frac{\alpha}{n}(\mathbf{d}_{j}^{\boldsymbol{T}} \boldsymbol{X}^{(q)}(j,:)^{\boldsymbol{T}}) \mathbf{e} + (1-\alpha)\|\boldsymbol{X}^{(q)}(j,:)\|_{1}\mathbf{v},
\end{equation}
for $j=1,2,\ldots,n$. By Theorem \ref{lem-2}, the solution of \eqref{3.4} exists and is unique, and the power method converges as $0<\alpha<1$ \cite{WJ2009A}.

However, it is required to store the {\it dense} higher-order PageRank tensor $\mathcal{X}$ in the higher-order PageRank problem, which is prohibitive, especially for large-scale data sets. In order to partially overcome this difficulty, Gleich {\it et al.} \cite{Gleich2014Multilinear} considered the following multilinear PageRank vector due to Li and Ng \cite{Li2014A}, which is an alternative to the higher-order PageRank tensor.

\begin{definition}\cite[Definition 4.1]{Gleich2014Multilinear}
Let $\mathcal{P}$ be an order-$m$, $n$-dimensional transition tensor representing an order-$(m-1)$ Markov chain, $\alpha \in (0,1)$ be a probability, and $\mathbf{v} \in \mathbb{R}^{n}$ be a stochastic vector.
Then the multilinear PageRank vector is a stochastic solution of the following system of polynomial equations
\begin{equation}\label{2.3}
\mathbf{x}=\alpha \mathcal{P} \mathbf{x}^{(m-1)} + (1-\alpha)\mathbf{v},
\end{equation} 
or equivalently
\begin{equation}\label{eq29}
\mathbf{x} = \alpha \boldsymbol{R} (\underbrace{\mathbf{x} \otimes \cdot \cdot \cdot \otimes \mathbf{x}}_{m-1\ terms}) + (1-\alpha)\mathbf{v},
\end{equation}
where
\begin{equation}\label{eqn210}
\boldsymbol{R}=\left[\begin{array}{ccc|ccc|c|ccc}
\mathcal{P}_{111} & \cdots & \mathcal{P}_{1n1} & \mathcal{P}_{112} & \cdots & \mathcal{P}_{1n2} & \cdots & \mathcal{P}_{11n} & \cdots & \mathcal{P}_{1nn} \\
\mathcal{P}_{211} & \cdots & \mathcal{P}_{2n1} & \mathcal{P}_{212} & \cdots & \mathcal{P}_{2n2} & \cdots & \mathcal{P}_{21n} & \cdots & \mathcal{P}_{2nn} \\
\vdots & \ddots & \vdots & \vdots & \ddots & \vdots & \vdots & \vdots & \ddots & \vdots \\
\mathcal{P}_{n11} & \cdots & \mathcal{P}_{nn1} & \mathcal{P}_{n12} & \cdots & \mathcal{P}_{nn2} & \cdots & \mathcal{P}_{n1n} & \cdots & \mathcal{P}_{nnn}
\end{array}\right]
\end{equation}
is a column stochastic matrix of the flattened tensor $\mathcal{P}$ along the first index.
\end{definition}

The following result shows the existence and uniqueness of the multilinear PageRank vector.
\begin{theorem}\cite[Theorem 4.3]{Gleich2014Multilinear}\label{lem-2.5}
Let $\mathcal{P}$ be an order-$m$ stochastic tensor and $\mathbf{v}$ be a nonnegative vector. Then the multilinear PageRank equation
\begin{equation*}
\mathbf{x}=\alpha \mathcal{P} \mathbf{x}^{(m-1)}+(1-\alpha) \mathbf{v}
\end{equation*}
has a unique solution when $\alpha \in (0,\frac{1}{m-1})$.
\end{theorem}

From Theorem \ref{lem-2} and Theorem \ref{lem-2.5}, we see that the higher-order PageRank tensor always exists and is unique, while the multilinear PageRank vector is not. Recently, some theoretical results are established on the existence and uniqueness of the multilinear PageRank vector; see \cite{Li2017The, Wen2020Multilinear,Li2014A,Dongdong2019Relaxation} and the references therein.
However, the conditions for these results are often complicated and some of them are difficult to use in practice.

On the other hand, the approximation of the multilinear PageRank vector to the high-order PageRank tensor can be poor.
For instance, if $m=3$, the approximation ${\bf x}{\bf x}^{\boldsymbol{T}}$ obtained from multilinear PageRank is a rank-$1$ {\it symmetric} matrix, and it is a poor approximation to a two-order PageRank tensor $\mathcal{X}$ that is a {\it nonsymmetric} matrix.
Thus, it is interesting to solve the higher-order PageRank problem \eqref{2.2} directly, provided that both the computational overhead and storage requirements can be reduced significantly. In this work, we consider how to solve the higher-order PageRank problem efficiently.

\section{The truncated power method for higher-order PageRank}\label{TPM}
In order to reduce the heavy storage requirements of the higher-order PageRank tensor, Ding {\it et al.} proposed a truncated power method for the higher-order PageRank problem \cite{Ding2018Fast}.
In this method, the following systems were considered instead of \eqref{3.2}:
\begin{equation}\label{3.3}
\boldsymbol{X}(:,j)=\alpha \mathbf{P}_{j} \boldsymbol{X}(j,:)^{\boldsymbol{T}}+(1-\alpha) \mathbf{G}(:,j),\quad j=1,2,\ldots,n,
\end{equation}
where $\mathbf{G} \in \mathbb{R}^{n \times n}$ is a nonnegative matrix with $\|\mathbf{G}\|_{l_{1}}=1$. Notice that this model \eqref{3.3} is
a modification to the original higher-order PageRank model \eqref{3.2}, and they are not mathematically to each other.

The key idea of the truncated power method is to approximate the stationary distribution matrix $\boldsymbol{X}$ by using $\alpha (\boldsymbol{S} + \mathbf{e}\boldsymbol{u}^{\boldsymbol{T}}) + (1-\alpha) \mathbf{G}$, where $\boldsymbol{S}$ is expected to be a {\it sparse} matrix containing those significantly large entries in $\boldsymbol X$, and $\boldsymbol{u}$ is a vector consisting of the background values for each column of $\boldsymbol{X}$.
Thus, the approximation obtained from the truncated power method is composed of a sparse matrix $\boldsymbol{S}$ and a vector $\boldsymbol{u}$ of dimension $n$, rather than an $n$-by-$n$ {\it dense} matrix.
Consequently, the storage requirements can be released significantly. 

More precisely, let $\mathbf{A} \in \mathbb{R}_{+}^{n \times n}$ and $\mathbf{b} \in \mathbb{R}_{+}^{n}$, i.e., $\mathbf{A}$ and $\mathbf{b}$ are nonnegative matrix and vector, respectively. The following operators were introduced \cite{Ding2018Fast}
\begin{itemize}[leftmargin=3em]
\item $\mathscr{P}: \mathbb{R}^{n \times n} \mapsto \mathbb{R}^{n \times n}$, $\mathscr{P}(\mathbf{A})(:,j):= \mathbf{P}_{j}\mathbf{A}(j,:)^{\boldsymbol{T}}$,
\item $\mathscr{T}: \mathbb{R}_{+}^{n} \mapsto \mathbb{R}_{+}^{n}$, $\mathscr{T}(\mathbf{b}):= \mathbf{s}^{*} + \mathbf{e} \mu^{*}$,
\item $\tilde{\mathscr{T}}: \mathbb{R}_{+}^{n \times n} \mapsto \mathbb{R}_{+}^{n \times n}$, $\tilde{\mathscr{T}}(\mathbf{A}):= [\mathscr{T}(\mathbf{A}(:,1)),\mathscr{T}(\mathbf{A}(:,2)),\ldots,\mathscr{T}(\mathbf{A}(:,n))]$,
\item $\mathscr{Q}: \mathbb{R}_{+}^{n \times n} \mapsto \mathbb{R}_{+}^{n \times n}$, $\mathscr{Q}(\mathbf{A}):= \alpha \tilde{\mathscr{T}}(\mathscr{P}(\mathbf{A})) + (1-\alpha)\mathbf{G}$.
\end{itemize}

The operator $\mathscr{T}$ is associated with the thresholding method for solving the following subproblem \cite[Algorithm 1]{Ding2018Fast}
\begin{equation}\label{27}
\begin{array}{l}
(\mathbf{s}^{*},  \mu^{*}) =  \mathop{\arg\min} \{\frac{1}{2}\|\mathbf{s}+\mathbf{e} \mu-\mathbf{b}\|_{2}^{2}+\beta\|\mathbf{s}\|_{1}\} \\
\ \ \ \ \ \ \ \ \ \ \ \ \ \text { s.t. } \mathbf{s} \geq 0, \mu \geq 0,\\
\end{array}
\end{equation}
where $\beta$ is a positive number and $\beta\|\mathbf{s}\|_{1}$ serves as an penalty term for sparsity. It was shown that the above optimization problem has a unique solution with explicit form \cite[Theorem 3.3]{Ding2018Fast}. An algorithm for solving this problem is given as follows.

\begin{algorithm}[H]\label{Alg10}
\caption{The thresholding method for solving the subproblem \eqref{27} \cite{Ding2018Fast}}
\SetKwInOut{Input}{Input}\SetKwInOut{Output}{Output}
\Input{The non-negative vector $\mathbf{b}$ and the parameter $\beta$;}
\Output{The sparse vector $\mathbf{s}$ (the non-zero components) and the background value $\mu$;}
\BlankLine
Sort the components of $\mathbf{b}$ into $b_{i_{1}} \geq \cdots \geq b_{i_{r}} > \beta \geq b_{i_{r+1}}\geq\cdots\geq b_{i_{n}}$\;
Find $d \in \{0,1, \ldots, r\}$ such that $(n-d) b_{i_{d}}>\sum\limits_{j=d+1}^{n} b_{i_{j}}+ n \beta \geq(n-d) b_{i_{d+1}}$ ($b_{i_{0}}:=+\infty$)\;
$\mu=\frac{1}{n-d} \sum\limits_{j=d+1}^{n} b_{i_{j}}+\frac{k}{n-d} \cdot \beta$\;
$s_{i_{j}}=\left\{\begin{array}{ll}b_{i_{j}}-\alpha-\mu, & j=1,2, \ldots, d, \\ 0, & j=d+1, d+2, \ldots, n ;\end{array}\right.$\;
{\bf return} $\mathbf{s}$, $\mu$\;
\end{algorithm}

Consider the following sequence
\begin{equation}\label{eq3.3}
\boldsymbol{X}^{(q)} = \alpha (\boldsymbol{S}^{(q)} + \mathbf{e}(\boldsymbol{u}^{(q)})^{\boldsymbol{T}}) + (1-\alpha)\mathbf{G}, \ \ q = 0,1,\ldots,
\end{equation}
then the iterative scheme for the truncated power method can be rewritten as
$$
\boldsymbol{X}^{(q+1)} = \mathscr{Q}(\boldsymbol{X}^{(q)}), \quad q = 0,1,\ldots
$$
Notice that for $\forall q \in \mathbb{N}$, we have $\boldsymbol{S}^{(q)} \in \mathbb{R}^{n \times n}_{+}$, $\boldsymbol{u}^{(q)} \in \mathbb{R}^{n}_{+}$ and $\|\boldsymbol{S}^{(q)} + \mathbf{e}{\boldsymbol{u}^{(q)}}^{\boldsymbol{T}}\|_{l_{1}} = 1$.
If we denote by $\boldsymbol{Y}^{(q+1)} = \mathscr{P}(\boldsymbol{X}^{(q)})$, then
$$
\boldsymbol{X}^{(q+1)} = \alpha \tilde{\mathscr{T}}(\boldsymbol{Y}^{(q+1)}) + (1-\alpha)\mathbf{G}, \quad q = 0,1,\ldots
$$
The truncated power method is described as follows, for more details, refer to \cite{Ding2018Fast}.

\begin{algorithm}[H]\label{Alg2}
\caption{A truncated power method for higher-order PageRank \cite{Ding2018Fast}}
\SetKwInOut{Input}{Input}\SetKwInOut{Output}{Output}
\Input{$\mathbf{Q}_{j}$, $\mathbf{d}_{j}$, $j=1,2,\ldots, n$, $\boldsymbol{S}^{(0)}$, $\boldsymbol{u}^{(0)}$, $\mathbf{G}$, $\alpha$, $\beta$, $tol$;}
\Output{$\boldsymbol{S}$, $\boldsymbol{u}$;}
\BlankLine
{\bf Let $q=0$ and $res=1$}\;
\While {$res > tol$}
{
$res = 0$\;
\For{$j=1,2,\ldots,n$}
{
$\boldsymbol{Y}^{(q+1)}(:,j) = \mathbf{P}_{j} (\alpha(\boldsymbol{S}^{(q)}(j,:)^{\boldsymbol{T}} + \boldsymbol{u}^{(q)}) +(1-\alpha)\mathbf{G}(j,:)^{\boldsymbol{T}})$\footnote{We mention that there is a typo in Algorithm 2 of \cite{Ding2018Fast}, where $\mathbf{G}(:,j)$ should be replaced with $\mathbf{G}(j,:)^{\boldsymbol{T}}$; see \eqref{eq3.3}.};
$[\boldsymbol{S}^{(q+1)}(:,j), \boldsymbol{u}^{(q+1)}(j)] = thresholding(\boldsymbol{Y}^{(q+1)}(:,j),\beta)$;\quad\%~{\tt sparsing}\

$res = res + \|(\boldsymbol{S}^{(q+1)}(:,j) - \boldsymbol{S}^{(q)}(:,j)) + (\boldsymbol{u}^{(q+1)}(:,j) - \boldsymbol{u}^{(q)}(:,j))\|_{1}$\;
}
$q = q + 1$\;
}
$\boldsymbol{S} = \boldsymbol{S}^{(q)}$, $\boldsymbol{u} = \boldsymbol{u}^{(q)}$\;
\end{algorithm}
%

\section{Improving the truncated power method for high-order PageRank problem}\label{rt}
Denote by $\mathbf{z} = \alpha(\boldsymbol{S}^{(q)}(j,:)^{\boldsymbol{T}} + \boldsymbol{u}^{(q)}) +(1-\alpha)\mathbf{G}(j,:)^{\boldsymbol{T}}$.
In Step 5 of Algorithm \ref{Alg2}, the matrix-vector multiplication $\mathbf{P}_{j} \mathbf{z}$ can be computed by
$\mathbf{Q}_{j}\mathbf{z}+ \mathbf{N}_{j} \mathbf{z}$ \cite{Ding2018Fast},
where $\mathbf{N}_{j}$ contains the entries arising from the set of dangling states. More precisely, let $\mathbf{d}_{j},j=1,2,\ldots,n$, be the vectors from the dangling states, we have that
\begin{equation}\label{eqn4.1}
\boldsymbol{Y}^{(q+1)}(:,j) = \mathbf{Q}_{j}\mathbf{z} + \frac{1}{n}(\mathbf{d}_{j}^{\boldsymbol{T}} \mathbf{z}) \mathbf{e}.
\end{equation}
Therefore, one has to form and store the vectors $\{\mathbf{d}_{j}\}_{j=1}^{n}$ explicitly,
which is prohibitive as $n$ is large. In this section, we propose two new strategies to improve the performance of the truncated power method.

\subsection{A variation on the truncated power method}
Recall that \eqref{3.3} and \eqref{3.2} are not mathematically equivalent, and the choice of $\mathbf{G}$ is very important.
Indeed, we can get different approximate solutions with different choices of $\mathbf{G}$.
In \cite{Ding2018Fast}, the matrix $\mathbf{G}$ is set to be a matrix that shares the same sparsity as $\mathcal{P}$.
Obviously, this option is non-optimal, and we have to store an auxiliary matrix in practice.
Thus, it is interested in using an appropriate matrix $\mathbf{G}$.
In this section, we propose to use
\begin{equation}\label{eqn42}
\mathbf{G} = \frac{1}{n}\mathbf{v}\mathbf{e}^{\boldsymbol{T}}.
\end{equation}
An advantage is that there is no need to store an auxiliary matrix in the truncated power method any more.

Now we consider how to compute $\boldsymbol{Y}^{(q+1)}(:,j)$ without forming and storing $\mathbf{d}_{j}$ explicitly.
For any matrix $\mathbf{A} \in \mathbb{R}_{+}^{n \times n}$, we define the operators
\begin{itemize}[leftmargin=3em]
\item $\mathscr{G}: \mathbb{R}^{n \times n} \mapsto \mathbb{R}^{n \times n}$, $\mathscr{G}(\mathbf{A})(:,j):= \alpha \mathscr{P}(\mathbf{A})(:,j) + (1-\alpha)\mathbf{G}(:,j)$,
\item ${\mathscr{H}}: \mathbb{R}_{+}^{n \times n} \mapsto \mathbb{R}_{+}^{n \times n}$, ${\mathscr{H}}(\mathbf{A}):= \tilde{\mathscr{T}}(\mathscr{G}(\mathbf{A}))$,
\end{itemize}
and consider the sequence
\begin{equation*}
\tilde{\boldsymbol{X}}^{(q)} = \tilde{\boldsymbol{S}}^{(q)} + \mathbf{e} (\tilde{\boldsymbol{u}}^{(q)})^{\boldsymbol{T}},  \ \ q = 0,1,\ldots
\end{equation*}
Then the iterative scheme of the improved truncated power method can be written as
\begin{equation*}
\tilde{\boldsymbol{X}}^{(q+1)} = \mathscr{H}(\tilde{\boldsymbol{X}}^{(q)}),  \ \ q = 0,1,\ldots
\end{equation*}

Let $\tilde{\boldsymbol{Y}}^{(q+1)} = \mathscr{G}(\tilde{\boldsymbol{X}}^{(q)})$, from \eqref{3.3} and $\mathbf{P}_{j} = \mathbf{Q}_{j} + \frac{1}{n} \mathbf{e} \mathbf{d}_{j}^{\boldsymbol{T}}$, we obtain
\begin{equation}\label{5.3}
\tilde{\boldsymbol{Y}}^{(q+1)}(:,j) = \alpha (\mathbf{Q}_{j} \tilde{\boldsymbol{X}}^{(q)}(j,:)^{\boldsymbol{T}}) + \frac{\alpha}{n}(\mathbf{d}_{j}^{\boldsymbol{T}}\tilde{\boldsymbol{X}}^{(q)}(j,:)^{\boldsymbol{T}})\mathbf{e} + (1-\alpha)\mathbf{G}(:,j).
\end{equation}
Multiplying $\mathbf{e}^{\boldsymbol{T}}$ on both sides of \eqref{5.3} yields
\begin{equation*}
\|\tilde{\boldsymbol{Y}}^{(q+1)}(:,j)\|_{1} = \alpha\|\mathbf{Q}_{j} \tilde{\boldsymbol{X}}^{(q)}(j,:)^{\boldsymbol{T}}\|_{1} + \alpha (\mathbf{d}_{j}^{\boldsymbol{T}}\tilde{\boldsymbol{X}}^{(q)}(j,:)^{\boldsymbol{T}}) + (1-\alpha) \|\mathbf{G}(:,j)\|_{1},
\end{equation*}
and thus
\begin{equation}\label{5.4}
\mathbf{d}_{j}^{\boldsymbol{T}}\tilde{\boldsymbol{X}}^{(q)}(j,:)^{\boldsymbol{T}} = \frac{\|\tilde{\boldsymbol{Y}}^{(q+1)}(:,j)\|_{1} - \alpha \|\mathbf{Q}_{j} \tilde{\boldsymbol{X}}^{(q)}(j,:)^{\boldsymbol{T}}\|_{1} - (1-\alpha)\|\mathbf{G}(:,j)\|_{1}}{\alpha}.
\end{equation}
Moreover, we have from $\tilde{\boldsymbol{Y}}^{(q+1)}(:,j)= \alpha \mathbf{P}_{j}\tilde{\boldsymbol{X}}^{(q)}(j,:)^{\boldsymbol{T}} + (1-\alpha)\mathbf{G}(:,j)$ that
\begin{equation}\label{5.5}
\|\tilde{\boldsymbol{Y}}^{(q+1)}(:,j)\|_{1} = \alpha \|\tilde{\boldsymbol{X}}^{(q)}(j,:)^{\boldsymbol{T}}\|_{1} + (1-\alpha)\|\mathbf{G}(:,j)\|_{1}.
\end{equation}
Substituting \eqref{5.5} into \eqref{5.4} gives
\begin{equation}\label{5.6}
\mathbf{d}_{j}^{\boldsymbol{T}}\tilde{\boldsymbol{X}}^{(q)}(j,:)^{\boldsymbol{T}} = \|\tilde{\boldsymbol{X}}^{(q)}(j,:)^{\boldsymbol{T}}\|_{1} - \|\mathbf{Q}_{j} \tilde{\boldsymbol{X}}^{(q)}(j,:)^{\boldsymbol{T}}\|_{1}.
\end{equation}
So we have from \eqref{5.3} that
\begin{align}\label{eqn4.7}
& \tilde{\boldsymbol{Y}}^{(q+1)}(:,j)  \nonumber \\
& =  \alpha (\mathbf{Q}_{j} \tilde{\boldsymbol{X}}^{(q)}(j,:)^{\boldsymbol{T}}) + \frac{\alpha}{n}(\|\boldsymbol{X}^{(q)}(j,:)^{\boldsymbol{T}}\|_{1} - \|\mathbf{Q}_{j} \tilde{\boldsymbol{X}}^{(q)}(j,:)^{\boldsymbol{T}}\|_{1})\mathbf{e} + (1-\alpha)\mathbf{G}(:,j) \nonumber \\
& = \alpha (\mathbf{Q}_{j} \tilde{\boldsymbol{X}}^{(q)}(j,:)^{\boldsymbol{T}}) + \frac{\alpha}{n}(\|\boldsymbol{X}^{(q)}(j,:)^{\boldsymbol{T}}\|_{1} - \|\mathbf{Q}_{j} \tilde{\boldsymbol{X}}^{(q)}(j,:)^{\boldsymbol{T}}\|_{1})\mathbf{e} + (1-\alpha)\mathbf{v}.
\end{align}

In summary, we have the following algorithm. The differences with Algorithm \ref{Alg2} are that we replace \eqref{eqn4.1} with \eqref{eqn4.7}, and there is no need to form and store the vectors $\mathbf{d}_{j},j=1,2,\ldots,n$, moreover, it is unnecessary to store an additional matrix $\mathbf{G}$; see Step 6 in Algorithm \ref{Alg3}.

\begin{algorithm}[H]
\caption{A variation on the truncated power method for higher-order PageRank}\label{Alg3}
\SetKwInOut{Input}{Input}\SetKwInOut{Output}{Output}
\Input{$\mathbf{Q}_{j}$, $j=1,2,\ldots, n$, $\tilde{\boldsymbol{S}}^{(0)}$, $\tilde{\boldsymbol{u}}^{(0)}$, $\mathbf{v}$, $\alpha$, $\beta$, $tol$;}
\Output{$\tilde{\boldsymbol{S}}$, $\tilde{\boldsymbol{u}}$;}
\BlankLine
{\bf Let $q=0$ and $res=1$}\;
\While {$res > tol$}
{
$res = 0$\;
\For{$j=1,2,\ldots,n$}
{
$\mathbf{y} = \tilde{\boldsymbol{S}}^{(q)}(j,:)^{\boldsymbol{T}} + \tilde{\boldsymbol{u}}^{(q)}$\;
$\tilde{\boldsymbol{Y}}^{(q+1)}(:,j) = \alpha (\mathbf{Q}_{j} \mathbf{y}) + \frac{\alpha}{n}(\|\mathbf{y}\|_{1} - \|\mathbf{Q}_{j} \mathbf{y}\|_{1})\mathbf{e} + \frac{1-\alpha}{n}\mathbf{v}$\;
$[\tilde{\boldsymbol{S}}^{(q+1)}(:,j), \tilde{\boldsymbol{u}}^{(q+1)}(j)] = thresholding(\tilde{\boldsymbol{Y}}^{(q+1)}(:,j),\beta)$\quad\%~{\tt sparsing}\

$res = res + \|(\tilde{\boldsymbol{S}}^{(q+1)}(:,j) - \tilde{\boldsymbol{S}}^{(q)}(:,j)) + (\tilde{\boldsymbol{u}}^{(q+1)}(:,j) - \tilde{\boldsymbol{u}}^{(q)}(:,j))\|_{1}$\;
}
$q = q + 1$\;
}
$\tilde{\boldsymbol{S}} = \tilde{\boldsymbol{S}}^{(q)}$, $\tilde{\boldsymbol{u}} = \tilde{\boldsymbol{u}}^{(q)}$\;
\end{algorithm}



Next, we consider the convergence of Algorithm \ref{Alg3}.
We first need the following lemma.
\begin{lemma}\cite[Theorem 3.4]{Ding2018Fast}\label{lem-5.1}
The operator $\mathscr{T} : \mathbb{R}_{+}^{n} \longmapsto \mathbb{R}_{+}^{n}$ is non-expansive, i.e., $\|\mathscr{T}(\mathbf{a}) - \mathscr{T}({\mathbf{b}})\|_{1} \leq \|\mathbf{a} - \mathbf{b}\|_{1}$.
\end{lemma}

The following result  was mentioned in \cite[pp.77]{Ding2018Fast} without proof.
We give a proof here.
\begin{lemma}\label{lem-5.2}
The operator $\tilde{\mathscr{T}} : \mathbb{R}_{+}^{n \times n} \longmapsto \mathbb{R}_{+}^{n \times n}$ is non-expansive, i.e. $\|\tilde{\mathscr{T}}(\mathbf{A}) - \tilde{\mathscr{T}}(\mathbf{B})\|_{l_{1}} \leq \|\mathbf{A} - \mathbf{B}\|_{l_{1}}$.
\end{lemma}
\begin{proof}
We have that
\begin{equation}\label{5.2.1}
\|\tilde{\mathscr{T}}(\mathbf{A}) - \tilde{\mathscr{T}}(\mathbf{B})\|_{l_{1}} = \sum_{j=1}^{n}\|[\mathscr{T}(\mathbf{A}(:,j)) - \mathscr{T}(\mathbf{B}(:,j))\|_{1}.
\end{equation}
According to Lemma \ref{lem-5.1},
\begin{equation}\label{5.2.2}
\|[\mathscr{T}(\mathbf{A}(:,j)) - \mathscr{T}(\mathbf{B}(:,j))\|_{1} \leq \|\mathbf{A}(:,j) - \mathbf{B}(:,j)\|_{1}.
\end{equation}
Combining \eqref{5.2.1} and \eqref{5.2.2}, we get
\begin{equation*}
\|\tilde{\mathscr{T}}(\mathbf{A}) - \tilde{\mathscr{T}}(\mathbf{B})\|_{l_{1}} \leq \sum_{j=1}^{n} \|\mathbf{A}(:,j) - \mathbf{B}(:,j)\|_{1}
= \|\mathbf{A} - \mathbf{B}\|_{l_{1}},
\end{equation*}
which completes the proof.
\end{proof}

The following theorem shows the convergence of Algorithm \ref{Alg3}.
\begin{theorem}\label{the-5.6}
When $\alpha \in (0,1)$, the iteration $\tilde{\boldsymbol{X}}_{q+1} = \mathscr{H}(\tilde{\boldsymbol{X}}_{q})$ converges to the unique fixed point in $\mathbb{R}_{+}^{n \times n}$ with an arbitrary initial point $\tilde{\boldsymbol{X}}_{0} \in \mathbb{R}_{+}^{n \times n}$.
\end{theorem}
\begin{proof}
Recall that
$\mathscr{H}(\mathbf{A})= \tilde{\mathscr{T}}(\mathscr{G}(\mathbf{A}))$, so we have from Lemma \ref{lem-5.2} that
\begin{equation*}
\|\mathscr{H}(\mathbf{A}) - \mathscr{H}(\mathbf{B})\|_{l_{1}} = \|\tilde{\mathscr{T}}(\mathscr{G}(\mathbf{A})) - \tilde{\mathscr{T}}(\mathscr{G}(\mathbf{B}))\|_{l_{1}} \leq \|\mathscr{G}(\mathbf{A}) - \mathscr{G}(\mathbf{B})\|_{l_{1}}.
\end{equation*}
Moreover,
\begin{equation*}
\|\mathscr{G}(\mathbf{A}) - \mathscr{G}(\mathbf{B})\|_{l_{1}} = \alpha \|\mathscr{P}(\mathbf{A}) - \mathscr{P}(\mathbf{B})\|_{l_{1}} \leq \alpha \|\mathbf{A} - \mathbf{B}\|_{l_{l}},
\end{equation*}
where we used the fact that the $l_{1}$-norm of the linear operator $\mathscr{P}$ is $1$ \cite{Ding2018Fast}.
Thus,
$$
\|\mathscr{H}(\mathbf{A}) - \mathscr{H}(\mathbf{B})\|_{l_{1}} \leq \alpha \|\mathbf{A} - \mathbf{B}\|_{l_{l}}<\|\mathbf{A} - \mathbf{B}\|_{l_{l}},
$$
and the conclusion follows from the Banach fixed point theorem \cite{CiarletLNFA}.
\end{proof}

\subsection{A truncated power method with partial updating}
In each iteration of Algorithm \ref{Alg3}, we have to update all the columns $\tilde{\boldsymbol{S}}^{(q+1)}(:,j),j=1,2,\ldots,n$,
and the workload is prohibitive as $n$ is large.
In practice, however, we are often concerned with several webpages with {\it top} PageRank values.
To release the overhead, we propose to {\it partially update} the columns of $\tilde{\boldsymbol{S}}^{(q+1)}$.
In other words, we only update some ``important" columns of the approximation in each iteration.

Let us discuss it in more detail.
Denote by
\begin{equation}\label{eqn410}
PV^{(q)}(j) = \sum_{i=1}^{n}\boldsymbol{S}^{(q)}_{ij} + \boldsymbol{u}^{(q)}_{j},\quad  q = 0,1,\ldots
\end{equation}
the {\it PageRank value} of the $j$-th webpage in the $q$-th iteration \cite{Ding2018Fast}. In the partially updating strategy, we first run $\tilde{\ell}$ (say, 1 or 2) iterations of Algorithm \ref{Alg3}.
Suppose that we are interested in the top $\varsigma$ PageRank values, and
let $\tau$ be the percentage of columns extracted from the previous iteration.
Let $\Omega^{(q)}$ be the index set of columns that is required to update during iterations of the new algorithm.
Next we consider how to update this set efficiently.
\begin{enumerate}[leftmargin=3em]
\item During the first $\tilde{\ell}$ iterations, we have to update all the columns, and $\Omega^{(q)} = \{1, 2, \ldots, n\}$ for $q = 0, 1, \ldots, \tilde{\ell}-1$, where $\Omega^{(q)}$ represents the set of columns need to be updated in the $(q+1)$-th iteration.
\item From the $(\tilde{\ell} + 1)$-th iteration, we only need to update the columns with top $card(\Omega^{(q)})$ PageRank values in the $q$-th iteration, where
    \begin{equation}\label{eqn411}
    card(\Omega^{(q)}) = \max\{\lfloor \tau \cdot card(\Omega^{(q-1)}) \rfloor, \varsigma\}.
    \end{equation}
\end{enumerate}

Based on the above discussions, we have the following algorithm that partially update the columns of the approximate solutions during iterations. If $\tau = 100\%$, then Algorithm \ref{Alg5} reduces to Algorithm \ref{Alg3}.

\begin{algorithm}[H]
\caption{A truncated power method with partial updating for higher-order PageRank }\label{Alg5}
\SetKwInOut{Input}{Input}\SetKwInOut{Output}{Output}
\Input{$\mathbf{Q}_{j}$, $j=1,2,\ldots, n$, $\hat{\boldsymbol{S}}^{(0)}$, $\hat{\boldsymbol{u}}^{(0)}$, $\alpha$, $\beta$, $\varsigma$, $\tau$, $\tilde{\ell}$ and $tol$;}
\Output{$\hat{\boldsymbol{S}}$, $\hat{\boldsymbol{u}}$;}
\BlankLine
{\bf Let $q=0$ and $res=1$}\;
\While {$res > tol$}
{
$res = 0$\;
\eIf{ $q \leq \tilde{\ell}-1$ }
{$\Omega^{(q)}$ = \{1, 2, \ldots, $n$\}\;}
{{\bf Determine} $card(\Omega^{(q)})$ via \eqref{eqn411} and {\bf select the column indexes corresponding to the top $card(\Omega^{(q)})$ elements in $\{{PV^{(q)}(j)}\}_{ j \in \Omega^{(q-1)}}$ to form $\Omega^{(q)}$ }\quad\%~{\tt determine the index set}\;
}
\For{$j = \in \Omega^{(q)}$}
{
$\mathbf{y} = \hat{\boldsymbol{S}}^{(q)}(j,:)^{\boldsymbol{T}} + \hat{\boldsymbol{u}}^{(q)}$\;
$\hat{\boldsymbol{Y}}^{(q+1)}(:,j) = \alpha (\mathbf{Q}_{j} \mathbf{y}) + \frac{\alpha}{n}(\|\mathbf{y}\|_{1} - \|\mathbf{Q}_{j} \mathbf{y}\|_{1})\mathbf{e}
+ \frac{1-\alpha}{n}\mathbf{v}$\;
$[\hat{\boldsymbol{S}}^{(q+1)}(:,j), \hat{\boldsymbol{u}}^{(q+1)}(j)] = thresholding(\hat{\boldsymbol{Y}}^{(q+1)}(:,j),\beta)$\quad\%~{\tt sparsing}\
{\bf Compute $PV^{(q+1)}(j)$ by using \eqref{eqn410}}\;
$res = res + \|(\hat{\boldsymbol{S}}^{(q+1)}(:,j) - \hat{\boldsymbol{S}}^{(q)}(:,j)) + (\hat{\boldsymbol{u}}^{(q+1)}(:,j) - \hat{\boldsymbol{u}}^{(q)}(:,j))\|_{1}$\;
}

$q = q + 1$\;
}
$\hat{\boldsymbol{S}} = \hat{\boldsymbol{S}}^{(q)}$, $\hat{\boldsymbol{u}} = \hat{\boldsymbol{u}}^{(q)}$\;
\end{algorithm}


The theorem as follows shows the convergence of Algorithm \ref{Alg5}.

\begin{theorem}\label{Thm4.4}
Consider the following iteration in Algorithm \ref{Alg5}:
\begin{equation}\label{Pu1}
\hat{\boldsymbol{X}}^{(q+1)}(:,j) =
\left\{
\begin{array}{lcl}
\hat{\boldsymbol{X}}^{(q)}(:,j) & & j \in \Omega^{(0)} \backslash \Omega^{(q)},\\
\mathscr{T}(\alpha \mathbf{P}_{j}\hat{\boldsymbol{X}}^{(q)}(j,:)^{\boldsymbol{T}} + (1-\alpha)\mathbf{G}(:,j)) & & j \in \Omega^{(q)},
\end{array}
\right.
\end{equation}
where $\Omega^{(0)} = \{1,2,...,n\}$ and $q=0, 1, \ldots$
If $\alpha \in (0,1)$, then the iteration \eqref{Pu1} converges to the unique fixed point in $\mathbb{R}_{+}^{n \times n}$ with arbitrary initial point $\hat{\boldsymbol{X}}_{0} \in \mathbb{R}_{+}^{n \times n}$.
\end{theorem}
\begin{proof}
From \eqref{Pu1}, we obtain
\begin{equation*}
\begin{aligned}
& \|\hat{\boldsymbol{X}}^{(q+1)} -  \hat{\boldsymbol{X}}^{(q)}\|_{l_{1}}\\
= & \sum_{j=1}^{n}\|\hat{\boldsymbol{X}}^{(q+1)}(:,j) - \hat{\boldsymbol{X}}^{(q)}(:,j)\|_{1}\\
= & \sum_{j \in \Omega^{(q)}}\|\hat{\boldsymbol{X}}^{(q+1)}(:,j) - \hat{\boldsymbol{X}}^{(q)}(:,j)\|_{1} + \sum_{j \in \Omega^{(0)} \backslash \Omega^{(q)}}\|\hat{\boldsymbol{X}}^{(q+1)}(:,j) - \hat{\boldsymbol{X}}^{(q)}(:,j)\|_{1}\\
= & \sum_{j \in \Omega^{(q)}}\|\hat{\boldsymbol{X}}^{(q+1)}(:,j) - \hat{\boldsymbol{X}}^{(q)}(:,j)\|_{1}, \quad q = 0, 1, \ldots
\end{aligned}
\end{equation*}
From Step 7 of Algorithm \ref{Alg5}, we notice that $\Omega^{(q+1)} \subseteq \Omega^{(q)} \subseteq \cdots \subseteq \Omega^{(0)}$. Thus, if $j \in \Omega^{(q+1)}$, then $j \in \Omega^{(q)}$, and
\begin{equation}\label{eqn413}
{\small\begin{aligned}
& \|\hat{\boldsymbol{X}}^{(q+2)} -  \hat{\boldsymbol{X}}^{(q+1)}\|_{l_{1}}\\
= & \sum_{j \in \Omega^{(q+1)}}\|\mathscr{T}(\alpha \mathbf{P}_{j}\hat{\boldsymbol{X}}^{(q+1)}(j,:)^{\boldsymbol{T}} + (1-\alpha)\mathbf{G}(:,j)) - \mathscr{T}(\alpha \mathbf{P}_{j}\hat{\boldsymbol{X}}^{(q)}(j,:)^{\boldsymbol{T}} + (1-\alpha)\mathbf{G}(:,j))\|_{1}.
\end{aligned}}
\end{equation}
It is seen that $\hat{\boldsymbol{X}}^{(q+1)} \in \mathbb{R}^{n \times n}_{+}$ and $\hat{\boldsymbol{X}}^{(q)} \in \mathbb{R}^{n \times n}_{+}$ as $\hat{\boldsymbol{X}}^{(0)} \in \mathbb{R}^{n \times n}_{+}$. By Lemma \ref{lem-5.1} and \eqref{eqn413},
\begin{equation*}
\begin{aligned}
& \|\hat{\boldsymbol{X}}^{(q+2)} -  \hat{\boldsymbol{X}}^{(q+1)}\|_{l_{1}}\\
\leq & \sum_{j \in \Omega^{(q+1)}}\|(\alpha \mathbf{P}_{j}\hat{\boldsymbol{X}}^{(q+1)}(j,:)^{\boldsymbol{T}} + (1-\alpha)\mathbf{G}(:,j)) - (\alpha \mathbf{P}_{j}\hat{\boldsymbol{X}}^{(q)}(j,:)^{\boldsymbol{T}} + (1-\alpha)\mathbf{G}(:,j))\|_{1}\\
= & \sum_{j \in \Omega^{(q+1)}}\|\alpha \mathbf{P}_{j}\big((\hat{\boldsymbol{X}}^{(q+1)}(j,:)^{\boldsymbol{T}} - \hat{\boldsymbol{X}}^{(q)}(j,:)^{\boldsymbol{T}}\big)\|_{1}
\leq \alpha \sum_{j \in \Omega^{(q+1)}}\|\hat{\boldsymbol{X}}^{(q+1)}(j,:)^{\boldsymbol{T}} - \hat{\boldsymbol{X}}^{(q)}(j,:)^{\boldsymbol{T}}\|_{1}\\
\leq & \alpha \sum_{j \in \Omega^{(q+1)}}\|\hat{\boldsymbol{X}}^{(q+1)}(j,:)^{\boldsymbol{T}} - \hat{\boldsymbol{X}}^{(q)}(j,:)^{\boldsymbol{T}}\|_{1} + \alpha \sum_{j \in \Omega^{(0)} \backslash \Omega^{(q+1)}}\|\hat{\boldsymbol{X}}^{(q+1)}(j,:)^{\boldsymbol{T}} - \hat{\boldsymbol{X}}^{(q)}(j,:)^{\boldsymbol{T}}\|_{1}\\
= & \alpha \sum_{j=1}^{n}\|\hat{\boldsymbol{X}}^{(q+1)}(j,:)^{\boldsymbol{T}} - \hat{\boldsymbol{X}}^{(q)}(j,:)^{\boldsymbol{T}}\|_{1}
= \alpha \|\hat{\boldsymbol{X}}^{(q+1)} - \hat{\boldsymbol{X}}^{(q)}\|_{l_{1}}.
\end{aligned}
\end{equation*}
So we complete the proof by the Banach fixed point theorem \cite{CiarletLNFA} and the fact that $0<\alpha<1$.
\end{proof}

\section{A sparse power method with partial updating for higher-order PageRank problem}\label{irt}
In Section \ref{rt}, we take $\mathbf{G}$ to be a fixed matrix.
However, similar to the truncated power method proposed in \cite{Ding2018Fast}, the problem \eqref{3.3} is not mathematically equivalent to the original problem \eqref{3.2}, either. In this section, we pay special attention to the original high-order PageRank problem \eqref{3.2} and consider how to solve it by using a sparse power method.

For any matrix $\mathbf{A} \in \mathbb{R}_{+}^{n \times n}$, we define the operator $\mathscr{Z}: \mathbb{R}_{+}^{n \times n} \mapsto \mathbb{R}_{+}^{n \times n}$ as
\begin{equation*}
 {\mathscr{Z}}(\mathbf{A}):= \tilde{\mathscr{T}}(\mathscr{W}(\mathbf{A})),
\end{equation*}
where the operator $\mathscr{W}$ is defined in \eqref{2.6}.
Then we construct the sequence
\begin{equation*}
\check{\boldsymbol{X}}^{(q)} = \check{\boldsymbol{S}}^{(q)} + \mathbf{e} (\check{\boldsymbol{u}}^{(q)})^{\boldsymbol{T}}, \ \ q = 0,1,\ldots,
\end{equation*}
then the truncated power method for \eqref{3.2} can be written as
\begin{equation*}
\check{\boldsymbol{X}}^{(q+1)} = \mathscr{Z}(\check{\boldsymbol{X}}^{(q)}), \ \ q = 0,1,\ldots
\end{equation*}

We present the algorithm as follows. Similar to \eqref{5.3}--\eqref{eqn4.7}, we have the following for the matrix-vector products
\begin{align}\label{eqn51}
 & \check{\boldsymbol{Y}}^{(q+1)}(:,j)  \nonumber \\
 & =   \alpha (\mathbf{Q}_{j} \check{\boldsymbol{X}}^{(q)}(:,j)^{\boldsymbol{T}}) + \frac{\alpha}{n}(\|\check{\boldsymbol{X}}^{(q)}(:,j)^{\boldsymbol{T}}\|_{1} - \|\mathbf{Q}_{j} \check{\boldsymbol{X}}^{(q)}(:,j)^{\boldsymbol{T}}\|_{1})\mathbf{e} +(1-\alpha) \|\check{\boldsymbol{X}}^{(q)}(:,j)^{\boldsymbol{T}}\|_{1}\mathbf{v}.
\end{align}

\begin{algorithm}[H]
\caption{A sparse power method for higher-order PageRank}\label{Alg4}
\SetKwInOut{Input}{Input}\SetKwInOut{Output}{Output}
\Input{$\mathbf{Q}_{j}$, $j=1,2,\ldots, n$, $\check{\boldsymbol{S}}^{(0)}$, $\check{\boldsymbol{u}}^{(0)}$, $\mathbf{v}$, $\alpha$, $\beta$, $tol$;}
\Output{$\check{\boldsymbol{S}}$, $\check{\boldsymbol{u}}$;}
\BlankLine
{\bf Let $q=0$ and $res=1$}\;
\While {$res > tol$}
{
$res = 0$\;
\For{$j=1,2,\ldots,n$}
{
$\mathbf{y} = \check{\boldsymbol{S}}^{(q)}(j,:)^{\boldsymbol{T}} + \check{\boldsymbol{u}}^{(q)}$\;
$\check{\boldsymbol{Y}}^{(q+1)}(:,j) = \alpha (\mathbf{Q}_{j} \mathbf{y}) + \frac{\alpha}{n}(\|\mathbf{y}\|_{1} - \|\mathbf{Q}_{j} \mathbf{y}\|_{1})\mathbf{e} +(1-\alpha) \|\mathbf{y}\|_{1}\mathbf{v}$\;
$[\check{\boldsymbol{S}}^{(q+1)}(:,j), \check{\boldsymbol{u}}^{(q+1)}(j)] = thresholding(\check{\boldsymbol{Y}}^{(q+1)}(:,j),\beta)$\quad\%~{\tt sparsing}\
$res = res + \|(\check{\boldsymbol{S}}^{(q+1)}(:,j) - \check{\boldsymbol{S}}^{(q)}(:,j)) + (\check{\boldsymbol{u}}^{(q+1)}(:,j) - \check{\boldsymbol{u}}^{(q)}(:,j))\|_{1}$\;
}
$q = q + 1$\;
}
$\check{\boldsymbol{S}} = \check{\boldsymbol{S}}^{(q)}$, $\check{\boldsymbol{u}} = \check{\boldsymbol{u}}^{(q)}$\;
\end{algorithm}
\begin{remark}
In Algorithm \ref{Alg4}, the original high-order PageRank problem \eqref{3.2} is solved instead of \eqref{3.3}, we expect that the approximation obtained from Algorithm \ref{Alg4} is more accurate than those from Algorithm \ref{Alg2} and Algorithm \ref{Alg3}. The key of Algorithm \ref{Alg4} is to apply the sparse operator $\tilde{\mathscr{T}}$ to the matrix obtained in Step 6, and we call it a sparse power method for higher-order PageRank. Analogous to Algorithm \ref{Alg3}, there is no need to form and store the vectors $\{{\bf d}_{j}\}_{j=1}^n$, and we only need to store a vector $\mathbf{v}$ instead of a matrix $\mathbf{G}$. The difference is that the value before ${\bf v}$ is fixed in Algorithm \ref{Alg3}, while that is variable in Algorithm \ref{Alg4}; see Step 6 of the two algorithms.
\end{remark}

We are ready to discuss the convergence of Algorithm \ref{Alg4}. If there is a number $\eta \in (0,1)$  such that
$$
\|\check{\boldsymbol{X}}^{(q+2)} - \check{\boldsymbol{X}}^{(q+1)}\|_{l_{1}} \leq \eta \|\check{\boldsymbol{X}}^{(q+1)} - \check{\boldsymbol{X}}^{(q)}\|_{l_{1}},\quad q=0,1,\ldots
$$
in each iteration, then according to the Banach fixed point theorem \cite{CiarletLNFA}, Algorithm \ref{Alg4} converges.
Unfortunately, it seems difficult to prove it strictly. The reason is that the value before ${\bf v}$ is variable during iterations of Algorithm \ref{Alg4}; refer to \eqref{eqn51}.
This is unlike the proof of Theorem \ref{the-5.6}, where the value can be eliminated in two consecutive iterations.
To deal with this problem, the theorem as follows considers the convergence of Algorithm \ref{Alg4} from a probabilistic point of view.
\begin{theorem}\label{theorem5.4}
For any matrices $ \mathbf{A}, \mathbf{B} \in \mathbb{R}_{+}^{n \times n}$, suppose that the probability
\begin{equation}\label{eqn52}
{\rm Pr}(\|\tilde{\mathscr{T}}(\mathbf{A}) - \tilde{\mathscr{T}}(\mathbf{B})\|_{l_{1}} \leq \eta \|\mathbf{A} - \mathbf{B}\|_{l_{1}}) = \omega,
\end{equation}
where $0<\eta<1$, and $\omega$ is a probability related to $\eta$.
Then we have that
\begin{equation}
{\rm Pr}(\|\mathscr{Z}(\mathbf{A}) - \mathscr{Z}(\mathbf{B})\|_{l_{1}} \leq \eta \|\mathbf{A} - \mathbf{B}\|_{l_{1}}) \geq \omega.
\end{equation}
\end{theorem}
\begin{proof}
From the definition of the operator $\mathscr{Z}$, we have
\begin{equation*}
  \|\mathscr{Z}(\mathbf{A}) - \mathscr{Z}(\mathbf{B})\|_{l_{1}} = \|\tilde{\mathscr{T}}(\mathscr{W}(\mathbf{A})) - \tilde{\mathscr{T}}(\mathscr{W}(\mathbf{B}))\|_{l_{1}}.
\end{equation*}
If $\|\tilde{\mathscr{T}}(\mathbf{A}) - \tilde{\mathscr{T}}(\mathbf{B})\|_{l_{1}} \leq \eta \|\mathbf{A} - \mathbf{B}\|_{l_{1}}$, then
\begin{equation*}
\begin{aligned}
&\|\mathscr{Z}(\mathbf{A}) - \mathscr{Z}(\mathbf{B})\|_{l_{1}} \leq \eta \|\mathscr{W}(\mathbf{A}) - \mathscr{W}(\mathbf{B})\|_{l_{1}} \\
= & \eta \sum_{j=1}^{n}\|[\alpha\mathbf{P}_{j} + (1-\alpha)\tilde{\mathbf{V}}](\mathbf{A}(j,:)^{\boldsymbol{T}} - \mathbf{B}(j,:)^{\boldsymbol{T}}) \|_{1}\\
\leq & \eta \sum_{j=1}^{n}\|(\mathbf{A}(j,:)^{\boldsymbol{T}} - \mathbf{B}(j,:)^{\boldsymbol{T}}) \|_{1}= \eta \|\mathbf{A} - \mathbf{B}\|_{l_{1}}.
\end{aligned}
\end{equation*}

Denote by $C$ the event $\|\tilde{\mathscr{T}}(\mathbf{A}) - \tilde{\mathscr{T}}(\mathbf{B})\|_{l_{1}} \leq \eta \|\mathbf{A} - \mathbf{B}\|_{l_{1}}$ and by $D$ the event $\|\mathscr{Z}(\mathbf{A}) - \mathscr{Z}(\mathbf{B})\|_{l_{1}} \leq \eta \|\mathbf{A} - \mathbf{B}\|_{l_{1}}$.
We see that the probability ${\rm Pr}(D|C) = 1$.
From \eqref{eqn52} and the total probability formula \cite{CKL1979A}, we obtain
\begin{equation*}
{\rm Pr}(D) \geq {\rm Pr}(C){\rm Pr}(D|C) = {\rm Pr}(C) = \omega,
\end{equation*}
which completes the proof.
\end{proof}




\begin{figure}[h]
  \centering
  \includegraphics[height=8cm,width=13cm]{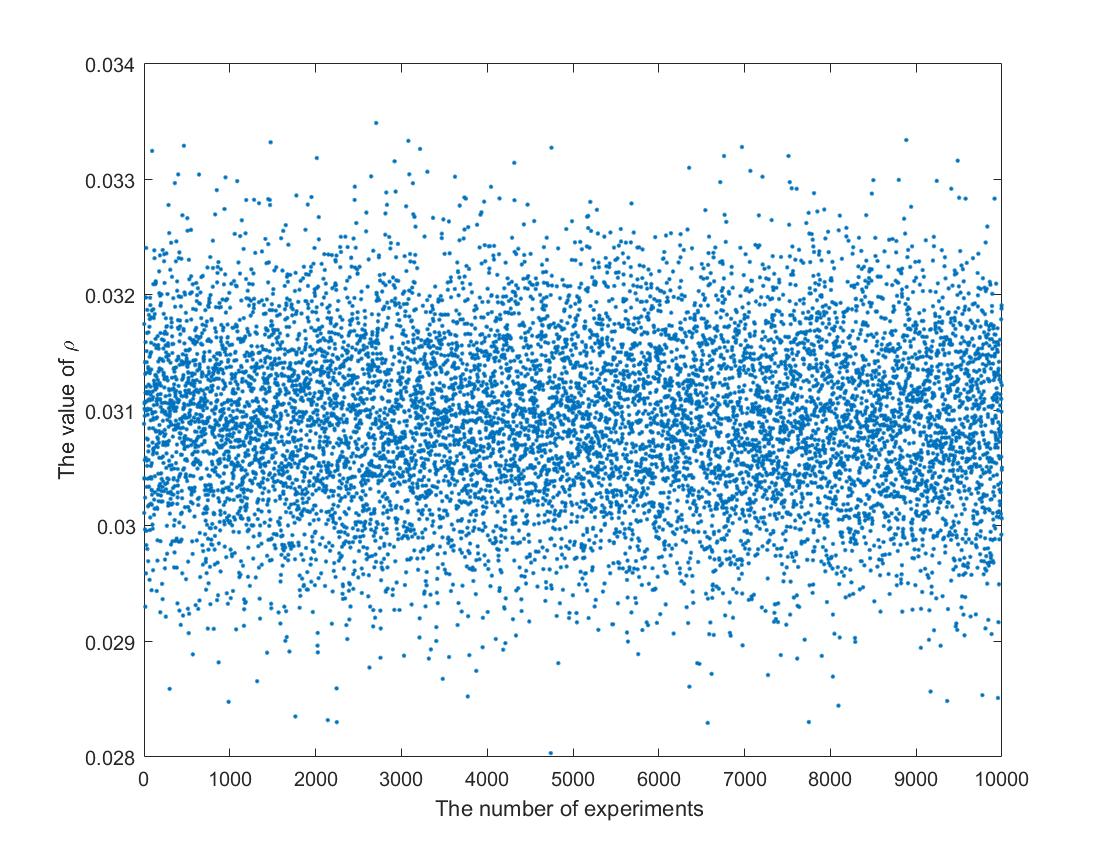}\\
  \caption{The values of $\rho=\frac{\|\tilde{\mathscr{T}}(\mathbf{A}) - \tilde{\mathscr{T}}(\mathbf{B})\|_{l_{1}}}{\|\mathbf{A} - \mathbf{B}\|_{l_{1}}}$ during iterations of Algorithm \ref{Alg4}, where we randomly generate two $1000 \times 1000$ non-negative matrices $\mathbf{A}$ and $\mathbf{B}$ for $10000$ times.} \label{fig.100}
\end{figure}

In general, the larger is $\eta$, the larger is $\omega$, and the probability of the event $D$ happens is higher.
Consequently, the probability of Algorithm \ref{Alg4} converges is higher.
However, we find experimentally that even if $\eta$ is relatively small, the probability of the event $D$ happens may still be high.
To show this more precisely,
we run Algorithm \ref{Alg4} with two $1000 \times 1000$ non-negative matrices $\mathbf{A}$ and $\mathbf{B}$
generated randomly for $10000$ times. The sparsity degree
is chosen as $\beta = \frac{1}{n^{2}}$, where $n$ is the size of the matrices.
In Figure \ref{fig.100}, we plot the 10000 values of the parameter denoted by
$$
\rho=\frac{\|\tilde{\mathscr{T}}(\mathbf{A}) - \tilde{\mathscr{T}}(\mathbf{B})\|_{l_{1}}}{\|\mathbf{A} - \mathbf{B}\|_{l_{1}}}.
$$
It is seen that in this experiment, if we take $\eta = 0.034$, then ${\rm Pr}(\|\tilde{\mathscr{T}}(\mathbf{A}) - \tilde{\mathscr{T}}(\mathbf{B})\|_{l_{1}} \leq \eta \|\mathbf{A} - \mathbf{B}\|_{l_{1}}) = 1$, which implies that the probability of Algorithm \ref{Alg4} converges can be very high. Experimentally, we find that Algorithm \ref{Alg4} works quite well in practice.

Finally, we apply the same trick of {\it partial updating} to Algorithm \ref{Alg4}, and get the following algorithm.

\begin{algorithm}[H]
\caption{A sparse power method with partial updating for higher-order PageRank}\label{Alg6}
\SetKwInOut{Input}{Input}\SetKwInOut{Output}{Output}
\Input{$\mathbf{Q}_{j}$, $j=1,2,\ldots, n$, $\bar{\boldsymbol{S}}^{(0)}$, $\bar{\boldsymbol{u}}^{(0)}$, $\alpha$, $\beta$, $\varsigma$, $\tau$, $\tilde{\ell}$, $tol$;}
\Output{$\bar{\boldsymbol{S}}$, $\bar{\boldsymbol{u}}$;}
\BlankLine
{\bf Let $q=0$ and $res=1$}\;
\While {$res > tol$}
{
$res = 0$\;
\eIf{ $q \leq \tilde{\ell}-1$ }
{$\Omega^{(q)}$ = \{1, 2, \ldots, $n$\}\;}
{{\bf Determine} $card(\Omega^{(q)})$ via \eqref{eqn411} and {\bf select the indexes corresponding to the top $card(\Omega^{(q)})$ elements in $\{{PV^{(q)}}(j)\}_{j \in \Omega^{(q-1)}}$ to form $\Omega^{(q)}$ }\quad\%~{\tt determine the index set}\;
}
\For{$j \in \Omega^{(q)}$}
{
$\mathbf{y} = \bar{\boldsymbol{S}}^{(q)}(j,:)^{\boldsymbol{T}} + \bar{\boldsymbol{u}}^{(q)}$\;
$\bar{\boldsymbol{Y}}^{(q+1)}(:,j) = \alpha (\mathbf{Q}_{j} \mathbf{y}) + \frac{\alpha}{n}(\|\mathbf{y}\|_{1} - \|\mathbf{Q}_{j} \mathbf{y}\|_{1})\mathbf{e} + (1-\alpha) \|\mathbf{y}\|_{1}\mathbf{v}$\;
$[\bar{\boldsymbol{S}}^{(q+1)}(:,j), \bar{\boldsymbol{u}}^{(q+1)}(j)] = thresholding(\bar{\boldsymbol{Y}}^{(q+1)}(:,j),\beta)$\quad\%~{\tt sparsing}\
{\bf Compute $PV^{(q+1)}(j)$ by using \eqref{eqn410}}\;
$res = res + \|(\bar{\boldsymbol{S}}^{(q+1)}(:,j) - \bar{\boldsymbol{S}}^{(q)}(:,j)) + (\bar{\boldsymbol{u}}^{(q+1)}(:,j) - \bar{\boldsymbol{u}}^{(q)}(:,j))\|_{1}$\;
}
$q = q + 1$\;
}
$\bar{\boldsymbol{S}} = \bar{\boldsymbol{S}}^{(q)}$, $\bar{\boldsymbol{u}} = \bar{\boldsymbol{u}}^{(q)}$\;
\end{algorithm}

The following theorem indicates that Algorithm \ref{Alg6} has a high probability of convergence. The proof is a combination of Theorem \ref{Thm4.4} and Theorem \ref{theorem5.4}, and thus is omitted.
\begin{theorem}
Consider the iterations of Algorithm \ref{Alg6}
\begin{equation}\label{Pu2}
\bar{\boldsymbol{X}}^{(q+1)}(:,j) =
\left\{
\begin{array}{lcl}
\mathscr{T}(\alpha \mathbf{P}_{j}\bar{\boldsymbol{X}}^{(q)}(j,:)^{\boldsymbol{T}} + (1-\alpha)\widetilde{\mathbf{V}}\bar{\boldsymbol{X}}^{(q)}(j,:)^{\boldsymbol{T}}) & & j \in \Omega^{(q)},\\
\bar{\boldsymbol{X}}^{(q)}(:,j) & & j \in \Omega^{(0)} \backslash \Omega^{(q)},
\end{array}
\right.
\end{equation}
where $\Omega^{(0)} = \{1,2,...,n\}$.
Given an arbitrary initial point $\bar{\boldsymbol{X}}_{0} \in \mathbb{R}_{+}^{n \times n}$, and for any matrices $ \mathbf{A}, \mathbf{B} \in \mathbb{R}_{+}^{n \times n}$, if
\begin{equation*}
{\rm Pr}(\|\tilde{\mathscr{T}}(\mathbf{A}) - \tilde{\mathscr{T}}(\mathbf{B})\|_{l_{1}} \leq \eta \|\mathbf{A} - \mathbf{B}\|_{l_{1}}) = \omega,
\end{equation*}
where $0<\eta<1$ and $\omega$ is a probability related to $\eta$.
Then we have
\begin{equation*}
{\rm Pr}(\|\boldsymbol{X}^{(q+2)} -  \boldsymbol{X}^{(q+1)}\|_{l_{1}} \leq \eta \|\boldsymbol{X}^{(q+1)} - \boldsymbol{X}^{(q)}\|_{l_{1}}) \geq \omega, \ \ \ q = 0, 1, \ldots
\end{equation*}
\end{theorem}


\section{Numerical experiments}\label{ne}
In this section, we perform numerical experiments to show the numerical performance and the efficiency of the proposed algorithms for high-order PageRank.
More precisely, we compare the proposed algorithms with some state-of-the art algorithms on real-world and synthetic data sets.
All the experiments are run on a 16 cores double Intel(R) Xeon(R) CPU E5-2637 v4, and with CPU 3.50 GHz and RAM 256GB. The operation system is 64-bit Windows 10, and the numerical results are obtained from running the MATLAB R2018b.

\subsection{On implementation of the matrix-vector products for large-scale multilinear PageRank}\label{Sec6.1}
We first discuss some details on implementation of the algorithms for multilinear PageRank.
Recall from \eqref{eq29} that the multilinear PageRank problem in a second-order Markov chain can be written as
\begin{equation}\label{6.1}
\mathbf{x} = \alpha \boldsymbol{R} (\mathbf{x} \otimes \mathbf{x}) + (1-\alpha)\mathbf{v},
\end{equation}
where the matrix $\boldsymbol{R}$ is defined in \eqref{eqn210}. In the experiments, we choose $\mathbf{v} = \frac{\mathbf{e}}{n}$, and $\mathbf{e}$ is the vector of all ones.
Let
\begin{equation}\label{eqn62}
\widetilde{\boldsymbol{Q}}=\left[\begin{array}{ccc|ccc|c|ccc}
\mathcal{Q}_{111} & \cdots & \mathcal{Q}_{1n1} & \mathcal{Q}_{112} & \cdots & \mathcal{Q}_{1n2} & \cdots & \mathcal{Q}_{11n} & \cdots & \mathcal{Q}_{1nn} \\
\mathcal{Q}_{211} & \cdots & \mathcal{Q}_{2n1} & \mathcal{Q}_{212} & \cdots & \mathcal{Q}_{2n2} & \cdots & \mathcal{Q}_{21n} & \cdots & \mathcal{Q}_{2nn} \\
\vdots & \ddots & \vdots & \vdots & \ddots & \vdots & \vdots & \vdots & \ddots & \vdots \\
\mathcal{Q}_{n11} & \cdots & \mathcal{Q}_{nn1} & \mathcal{Q}_{n12} & \cdots & \mathcal{Q}_{nn2} & \cdots & \mathcal{Q}_{n1n} & \cdots & \mathcal{Q}_{nnn}
\end{array}\right],
\end{equation}
then we have from \eqref{eqn210} and \eqref{eqn62} that
\begin{equation}\label{eqn66}
\boldsymbol{R} = \widetilde{\boldsymbol{Q}} + \frac{1}{n} \mathbf{e} \tilde{\boldsymbol{d}}^{\boldsymbol{T}},
\end{equation}
where $\tilde{\boldsymbol{d}} \in \mathbb{R}^{n^{2}}$ is the vector composed of the entries arising from the dangling states of $\widetilde{\boldsymbol{Q}}$.


Given $\mathbf{x}^{(0)} \in \mathbb{R}_{+}^{n}$ with $\|\mathbf{x}^{(0)}\|_{1} = 1$, the following sequence
\begin{equation}\label{6.2}
\mathbf{x}^{(q+1)} = \alpha \boldsymbol{R} (\mathbf{x}^{(q)} \otimes \mathbf{x}^{(q)}) + (1-\alpha)\mathbf{v},\quad q=0,1,\ldots,
\end{equation}
is the key step of the fixed-point iteration method proposed in \cite{Gleich2014Multilinear} for multilinear PageRank. The fixed-point iteration method is popular for the computation of the multilinear PageRank problem. For instance, based on the shifted fixed-point iteration, Cipolla {\it et al.} \cite{Cipolla2019Extrapolation} investigate restarted extrapolation to speed up this method. Meini {\it et al.} \cite{Meini2018Perron} point that an interpretation of the solution as a Perron eigenvector can allow us to devise new fixed-point algorithms for its computation.
Recently, the convergence of the fixed-point iteration is revisited in \cite{HW}.

In \cite{Cipolla2019Extrapolation}, Cipolla {\it et al.} applied extrapolation skills to the shifted fixed-point iteration (SFPM) and the inner-outer iteration (IOM) proposed in \cite{Gleich2014Multilinear}, and presented the ESFPM and EIOM methods relied on the simplified topological $\varepsilon$-algorithm in its restarted form.
As the framework of the two compared algorithms ESFPM and EIOM is based on the computation of \eqref{6.2} during iterations, we focus on
computing $\mathbf{x}^{(q+1)}$ for a given vector $\mathbf{x}^{(q)}$ in this subsection.

First, we consider how to compute $\mathbf{x}^{(q+1)}$ without storing the vector $\tilde{\boldsymbol{d}}$.
By \eqref{eqn66}, the relation \eqref{6.2} can be reformulated as
\begin{equation}\label{6.4}
\mathbf{x}^{(q+1)} = \alpha \widetilde{\boldsymbol{Q}} (\mathbf{x}^{(q)} \otimes \mathbf{x}^{(q)}) + \frac{\alpha}{n} \mathbf{e} \tilde{\boldsymbol{d}}^{\boldsymbol{T}} (\mathbf{x}^{(q)} \otimes \mathbf{x}^{(q)}) + (1-\alpha)\mathbf{v},
\end{equation}
which is a stochastic vector.
Multiplying on both sides of \eqref{6.4} by $\mathbf{e}^{\boldsymbol{T}}$, we get
\begin{equation*}
1 = \alpha \|\widetilde{\boldsymbol{Q}} (\mathbf{x}^{(q)} \otimes \mathbf{x}^{(q)})\|_{1} + \alpha \tilde{\boldsymbol{d}}^{\boldsymbol{T}} (\mathbf{x}^{(q)} \otimes \mathbf{x}^{(q)}) + (1-\alpha),
\end{equation*}
and
\begin{equation}\label{6.5}
\tilde{\boldsymbol{d}}^{\boldsymbol{T}} (\mathbf{x}^{(q)} \otimes \mathbf{x}^{(q)}) = 1 - \|\widetilde{\boldsymbol{Q}} (\mathbf{x}^{(q)} \otimes \mathbf{x}^{(q)})\|_{1}.
\end{equation}
A combination of \eqref{6.5} and \eqref{6.4} gives
\begin{equation}\label{6.6}
\mathbf{x}^{(q+1)} = \alpha \widetilde{\boldsymbol{Q}} (\mathbf{x}^{(q)} \otimes \mathbf{x}^{(q)}) + \frac{\alpha}{n} (1 - \|\widetilde{\boldsymbol{Q}} (\mathbf{x}^{(q)} \otimes \mathbf{x}^{(q)}\|_{1}) \mathbf{e} + (1-\alpha)\mathbf{v}.
\end{equation}

Second, we consider how to calculate the matrix-vector product $\widetilde{\boldsymbol{Q}} (\mathbf{x}^{(q)} \otimes \mathbf{x}^{(q)})$
without forming and storing the vector $\mathbf{x}^{(q)} \otimes \mathbf{x}^{(q)}$ of size $n^2$.
Partition the matrix $\widetilde{\boldsymbol{Q}}$ into the following form
\begin{equation}\label{eqn68}
\widetilde{\boldsymbol{Q}} = [\widetilde{\boldsymbol{Q}}_{1}, \widetilde{\boldsymbol{Q}}_{2},\ldots,\widetilde{\boldsymbol{Q}}_{n}],
\end{equation}
where $\widetilde{\boldsymbol{Q}}_{j} \in \mathbb{R}^{n \times n}, j=1,2,\ldots,n$.
Let $\mathbf{x}^{(q)} = [\mathbf{x}^{(q)}_{1}, \mathbf{x}^{(q)}_{2}, \ldots, \mathbf{x}^{(q)}_{n}]^{\boldsymbol{T}}$, where $\mathbf{x}^{(q)}_{i}$ is the $i$-th element of $\mathbf{x}^{(q)}$.
Notice that
\begin{equation*}
\mathbf{x}^{(q)} \otimes \mathbf{x}^{(q)} =
\left[\begin{array}{c}
\mathbf{x}^{(q)}_{1}\mathbf{x}^{(q)}\\
\mathbf{x}^{(q)}_{2}\mathbf{x}^{(q)}\\
\vdots\\
\mathbf{x}^{(q)}_{n}\mathbf{x}^{(q)}
\end{array}\right],
\end{equation*}
where $\mathbf{x}^{(q)}_{j}\mathbf{x}^{(q)} \in \mathbb{R}^{n},~ j=1,2,\ldots,n$.
Thus,
\begin{equation}\label{6.7}
\widetilde{\boldsymbol{Q}} (\mathbf{x}^{(q)} \otimes \mathbf{x}^{(q)}) = [\widetilde{\boldsymbol{Q}}_{1}, \widetilde{\boldsymbol{Q}}_{2},\ldots,\widetilde{\boldsymbol{Q}}_{n}]
\left[\begin{array}{c}
\mathbf{x}^{(q)}_{1}\mathbf{x}^{(q)}\\
\mathbf{x}^{(q)}_{2}\mathbf{x}^{(q)}\\
\vdots\\
\mathbf{x}^{(q)}_{n}\mathbf{x}^{q}
\end{array}\right]
= \sum_{j = 1}^{n} \mathbf{x}^{(q)}_{j} \widetilde{\boldsymbol{Q}}_{j} \mathbf{x}^{(q)}.
\end{equation}
Substituting \eqref{6.7} into \eqref{6.5} gives
\begin{equation}
\mathbf{x}^{(q+1)} = \alpha \sum_{j = 1}^{n} \mathbf{x}^{(q)}_{j} \widetilde{\boldsymbol{Q}}_{j} \mathbf{x}^{(q)} + \frac{\alpha}{n} (1 - \|\sum_{j = 1}^{n} \mathbf{x}^{(q)}_{j} \widetilde{\boldsymbol{Q}}_{j} \mathbf{x}^{(q)}\|_{1}) \mathbf{e} + (1-\alpha)\mathbf{v}.
\end{equation}
In the numerical experiments, we perform \eqref{6.2} in the above way. One merit is that one only needs to form and store the sparse matrices $\widetilde{\boldsymbol{Q}}_{j},j=1,2,\ldots,n$, instead of the dense matrix $\boldsymbol{R}$.

\subsection{Experiments on real-world data set}\label{sec-real}
In this section, we run the algorithms on a real-world data set.
This data is the {\tt .GOV Web collection} arising in 2002 TREC \cite{voorhees2005trec}, with $n=100000$ web pages being linked by 39255 anchor terms \footnote{We thank Dr. Weiyang Ding for providing us with the data set and the MATLAB codes of the truncated power method.}.
We can describe the web link information in this real-world data through a $100000 \times 100000 \times 39255$ tensor with elements
$$
\mathcal{A}(i, j, s)=\left\{\begin{array}{ll}1, & \text { if webpage } i \text { links to webpage } j \text { via anchor term } s, \\ 0, & \text { otherwise }, \end{array}\right.
$$
whose sparsity is about $1.22 \times 10^{-7}$ \%.

Recently, Wu et al. \cite{WuYubao} proposed to formulate second-order Markov chains by using edge-to-edge transition probabilities, and developed Monte Carlo methods to compute stationary probability distributions.
However, Ding {\it et al.} \cite{Ding2018Fast} point that this model has some defects.
In contrast, the model constructed in \cite{Ding2018Fast} might be able to apply directly to the higher-order Markov chains, whereas the edge-to-edge formulation would not be straightforward.
For the second-order Markov chain, we have \cite{Ding2018Fast}
\begin{equation*}
\begin{aligned}
& \operatorname{Pr}\left(Y_{t}=i \mid Y_{t-1}=j, Y_{t-2}=k\right)\\
&  = \sum_{s=1}^{m} \operatorname{Pr}\left(Y_{t}=i \mid N=s, Y_{t-1}=j\right) \cdot \operatorname{Pr}\left(N=s \mid Y_{t-1}=j, Y_{t-2}=k\right),
\end{aligned}
\end{equation*}
and the elements of the tensor $\mathcal{Q}$ turn out to be
\begin{equation*}
\mathcal{Q}(i, j, k)=\sum_{s=1}^{m}\left(\frac{\mathcal{A}(j, i, s)}{\sum_{i=1}^{n} \mathcal{A}(j, i, s)}\right) \cdot\left(\frac{\mathcal{A}(k, j, s)}{\sum_{s=1}^{m} \mathcal{A}(k, j, s)}\right).
\end{equation*}
Further, consider the matrices $\boldsymbol{U}_{j} \in \mathbb{R}^{n \times m}$ and $\boldsymbol{W}_{j} \in \mathbb{R}^{m \times n},j=1,2,\ldots,n$, with elements
\begin{equation}\label{r-1}
\boldsymbol{U}_{j}(i, s)=\frac{\mathcal{A}(j, i, s)}{\sum_{i=1}^{n} \mathcal{A}(j, i, s)} \quad \text { and } \quad \boldsymbol{W}_{j}(s, k)=\frac{\mathcal{A}(k, j, s)}{\sum_{s=1}^{m} \mathcal{A}(k, j, s)},
\end{equation}
respectively, then $\mathcal{Q}(i,j,k)$ can be reformulated as
\begin{equation*}
\mathbf{Q}_{j}(i, k)=\sum_{s=1}^{m} \boldsymbol{U}_{j}(i, s) \boldsymbol{W}_{j}(s, k),
\end{equation*}
which is the $(i,k)$-th element of the matrix
\begin{equation}\label{eqn611}
\mathbf{Q}_{j} = \boldsymbol{U}_{j} \boldsymbol{W}_{j},\quad j=1,2,\ldots,n.
\end{equation}
Indeed, the real data we get is not the tensor $\mathcal{A}$, but the position of non-zero elements in $\mathcal{A}$.
Therefore, in practical calculations, we only calculate $\{\boldsymbol{U}_{j}\}_{j=1}^{n}$ and $\{\boldsymbol{W}_{j}\}_{j=1}^{n}$ by using the information of the positions of non-zero elements of $\mathcal{A}$.
If $\mathcal{A}(j,i,s) = 0$, we just set $\boldsymbol{U}_{j}(i,s) = 0$, otherwise, we calculate $\boldsymbol{U}_{j}(i,s)$ by using \eqref{r-1}.
The computation of $\{\boldsymbol{W}_{j}\}_{j=1}^{n}$ is in a similar way. Thus, it is only necessary to form the matrix
\begin{equation}\label{eqn613}
\mathbf{Q} = [\mathbf{Q}_{1}, \mathbf{Q}_{2}, \ldots, \mathbf{Q}_{n}]
\end{equation}
for the high-order PageRank problem.

For the multilinear PageRank problem, we need to form the matrices $\{\widetilde{\boldsymbol{Q}}_{i}\}_{i=1}^n$; see \eqref{eqn68}.
However, it seems that there are no explicit formulae such as \eqref{eqn611} for $\widetilde{\boldsymbol{Q}}_i$. This is another advantage of
the higher-order PageRank model over the multilinear PageRank model.
There are two ways to settle this problem. The first one is to form the tensor $\mathcal{Q}$ directly, and then flatten it along the first index.
Indeed, it is seen that the two matrices $\widetilde{\boldsymbol{Q}}$ and $\mathbf{Q}$ are different only in the order of columns. Thus, the second way is to form $\mathbf{Q}$ firstly, and then get $\widetilde{\boldsymbol{Q}}$ by permutation.
In general, the second option will be much cheaper than the first one.
However, the second option is also prohibitive for large-scale problems.
For example, it took more than 24 hours to tansform the matrix $\mathbf{Q}$ to $\widetilde{\boldsymbol{Q}}$ in this experiment.
Therefore, in this experiment, we do not consider the multilinear PageRank problem for this real-world data set any more.

We run the power method, Algorithm \ref{Alg2} due to Ding {\it et al.} \cite{Ding2018Fast}, as well as the proposed Algorithm \ref{Alg3}--Algorithm \ref{Alg6} for this problem.
The damping factor is chosen as $\alpha=0.85, 0.90, 0.95, 0.99$, and the parameter $\beta$ for sparsing is chosen as $\frac{1}{n^{2}},\frac{1}{n^{3}}$ and $\frac{1}{n^{4}}$, respectively.
In Algorithm \ref{Alg2}, the matrix $\mathbf{G}$ is generated by using the MATLAB command {\tt sprand}.
As the sparsity of the tested matrix $\mathbf{Q}$ is about $1.84 \times 10^{-8} \%$, the sparsity of $\mathbf{G}$ is also chosen as $1.84 \times 10^{-8} \%$.
In Algorithm \ref{Alg5} and Algorithm \ref{Alg6}, we choose $\varsigma = 10$, $\tau = 0.1$ and $\tilde{\ell} = 1$.
The stopping criterion is
$$
\frac{\|\boldsymbol{X}^{(q+1)} - \boldsymbol{X}^{(q)}\|_{l_{1}}}{\| \boldsymbol{X}^{(q)}\|_{l_{1}}} \leq 1 \times 10^{-6}
$$
for all the algorithms.
We set the maximum number of iterations to be $20$.
So as to access the quality of the approximations, we denote by ``Error" the relative error between the ``computed" solution $\boldsymbol{\widetilde{X}}$ from the algorithms and the ``exact" solution $\boldsymbol{X}^{*}$:
$$
{\rm Error}=\frac{\|\boldsymbol{\widetilde{X}} - \boldsymbol{X}^{*} \|_{l_{1}}}{\|\boldsymbol{X}^{*}\|_{l_{1}}},
$$
where the ``exact" solution is obtained from running the power method on the original high-order PageRank problem \eqref{3.2} with stopping criterion $tol=10^{-6}$.
The numerical results of this experiment are shown in Appendix \ref{AppA}; see Table \ref{t-6}--Table \ref{t-7}, in which ``Iter" denotes the number of iterations, and ``CPU" represents the CPU time in seconds. Here the computation time is composed of that both for generating the matrices $\{\mathbf{Q}_{j}\}_{j=1}^{n}$ and for solving the high-order PageRank problem.

Some comments are in order.
First, we observe from Table \ref{t-6}--Table \ref{t-90} that the four proposed algorithms converge much faster than Algorithm \ref{Alg2}.
For instance, when $\alpha = 0.85$, Algorithm \ref{Alg3} and Algorithm \ref{Alg4} are about 75\% faster than Algorithm \ref{Alg2}, while Algorithm \ref{Alg5} and Algorithm \ref{Alg6} are about 90\% faster than Algorithm \ref{Alg2}.
Moreover, the superiority become more obvious as  the damping factor $\alpha$ increases. In particular, when $\alpha=0.99$, we see that the power method, Algorithm \ref{Alg2}, Algorithm \ref{Alg3} and Algorithm \ref{Alg4} fail to work, while Algorithm \ref{Alg5} and Algorithm \ref{Alg6} run quite well.
Second, the accuracy of the proposed algorithms is much higher than that of Algorithm \ref{Alg2}. More precisely, the approximations obtained from Algorithm \ref{Alg3} and Algorithm \ref{Alg4} are about five orders higher, while those from Algorithm \ref{Alg5} and Algorithm \ref{Alg6} are about four orders higher than Algorithm \ref{Alg2}. Specifically, the accuracy of the approximations obtained from Algorithm \ref{Alg4} is the highest. This is because that it solves the original higher-order PageRank model directly.

Third, it seems that the convergence speed of Algorithm \ref{Alg5} and Algorithm \ref{Alg6} has little to do with the damping factor $\alpha$. Indeed, in the two algorithms, the updated columns will be fewer and fewer as the iteration proceeds, and we benefit from the partial updating strategy.
Fourth, in order to demonstrate the effectiveness of the sparse methods with partial updating in more detail, we present in
Table \ref{t-7} the number of the top-ten ranking web pages obtained from the power method and Algorithm \ref{Alg2}--Algorithm \ref{Alg6}, respectively.
Although the order of ranking changes a little, it is seen that the top-ten web pages from the six algorithms are about the same, and the partial updating strategy proposed is favorable.

\subsection{Experiments on synthetic data set}\label{AD}
In this section, we show performance of the proposed algorithms on a synthetic data set that is generated randomly \cite{Ding2018Fast}.
Before running the algorithms, we stress that even for medium-sized problems, the cost of transforming the matrix $\mathbf{Q}$ (for high-order PageRank) to the matrix $\widetilde{\boldsymbol{Q}}$ (for multilinear PageRank) can be very high; refer to \eqref{eqn613} and \eqref{eqn68}.
In this subsection, the test matrices $\mathbf{Q}$ are constructed by using the MATLAB build-in function
\begin{equation}\label{eqn614}
\mathbf{Q}=sprand(n, n^{2}, sparsity).
\end{equation}
In the first experiment, we choose the size $n = 2000, 4000, 6000, 8000, 10000$, and $sparsity = 0.0001\%$.
Table \ref{t-1} shows the CPU time in seconds for transforming $\mathbf{Q}$ to $\widetilde{\boldsymbol{Q}}$. It is seen that the CPU time is overwhelming when $n$ is large.

\begin{table}[h]\caption{The CPU time for transforming $\mathbf{Q}$ to $\widetilde{\boldsymbol{Q}}$}\label{t-1}
\centering
\begin{tabular}{ccc}
  \toprule
  Sparsity of $\mathbf{Q}$ & $n$ & CPU (s)\\
  \midrule
  \multirow{5}{*}{0.0001\%}
  & 2000 & 11.57 \\
  & 4000 & 485.75 \\
  & 6000 & 3755.05 \\
  & 8000 & 15022.28 \\
  & 10000 & 48989.85 \\
  \bottomrule
\end{tabular}
\end{table}





In the second experiment, we choose $n = 10000$ and $sparsity = 0.0001\%$  and $0.00001\%$ in \eqref{eqn614}, respectively.
We compare Algorithm \ref{Alg3}--Algorithm \ref{Alg6} with eight algorithms, i.e., the power method and Algorithm \ref{Alg2} for high-order PageRank, and the following six algorithms for multilinear PageRank: 
\begin{itemize}[leftmargin=3em]
\item MNM \cite{Guo2018A}: The modified Newton method due to Guo {\it et al}.
\item Pb-Ns (Algorithm 2, Algorithm 3 and Algorithm 4 in \cite{Meini2018Perron}): The three Perron-based algorithms for multilinear PageRank due to Meini and Poloni.
\item ESFPM and EIOM (Algorithm 1 in \cite{Cipolla2019Extrapolation}): Two algorithms due to Cipolla {\it et al.}, which are based on the simplified topological $\varepsilon$-algorithm in their restarted form, to accelerate the SFPM and IOM algorithms.
\end{itemize}
In ESFPM and EIOM, we take the number of inner iterations as 2, and the parameter $\gamma$ is chosen as 1.
In Algorithm \ref{Alg5} and Algorithm \ref{Alg6}, we choose $\varsigma = 10$, $\tau = 0.1$ and $\tilde{\ell} = 1$.

Indeed, for synthetic data set, one can also generate $\widetilde{\boldsymbol{Q}}$ first, and then get $\mathbf{Q}$ by permutation.
For the sake of justification, we do not consider the time for matrix transformation any more, and the CPU time (in seconds) in the following tables only includes that for solving the high-order PageRank problem or the multilinear PageRank problem.

For the multilinear PageRank and the higher-order PageRank problems, we take advantage of
$$
\frac{\|\mathbf{x}^{(q+1)} - \mathbf{x}^{(q)}\|_{1}}{\|\mathbf{x}^{(q)}\|_{1}} \leq 1 \times 10^{-8}\quad{\rm and}\quad\frac{\|\boldsymbol{X}^{(q+1)} - \boldsymbol{X}^{(q)}\|_{l_{1}}}{\|\boldsymbol{X}^{(q)}\|_{l_{1}}} \leq 1\times 10^{-8}
$$
as the stopping criterion, respectively, where $\mathbf{x}^{(q)}$ and $\boldsymbol{X}^{(q)}$ stand for the approximations obtained from the $q$-th iteration of the algorithms, respectively.
Recall that we only need to store a sparse matrix $\boldsymbol{S}$ and a vector $\boldsymbol{u}$ in Algorithm \ref{Alg2}--Algorithm \ref{Alg6}, rather than an $n$-by-$n$ dense matrix.

Similarly, we denote by ``Error" the relative error between the ``computed" solution $\widetilde{\boldsymbol{X}}$ (which is $\widetilde{\boldsymbol{x}}\widetilde{\boldsymbol{x}}^T$ for multilinear PageRank problem) from the algorithms and the ``exact" solution $\boldsymbol{X}^{*}$:
$$
{\rm Error}=\frac{\|\widetilde{\boldsymbol{X}} - \boldsymbol{X}^{*} \|_{l_{1}}}{\|\boldsymbol{X}^{*}\|_{l_{1}}},
$$
where the ``exact" solution is obtained from running the power method on the original high-order PageRank problem \eqref{3.2} with stopping criterion $tol=10^{-12}$. The numerical results of the experiment are shown in Appendix \ref{AppB}; see Table \ref{t-2}--Table \ref{t-5}. In this experiment, we choose the damping factors as $\alpha=0,85,0.90,0.95,0.99$, and the parameters for sparsing as $\beta=\frac{1}{n^{2}},\frac{1}{n^{3}},\frac{1}{n^{4}}$, respectively.


Again, it is seen from Table \ref{t-2}--Table \ref{t-5} that Algorithm \ref{Alg3}--Algorithm \ref{Alg6} outperform Algorithm \ref{Alg2} both in terms of CPU time and accuracy. In particular, Algorithm \ref{Alg4} is the best in terms of accuracy, while Algorithm \ref{Alg6} converges the fastest as $sparsity=0.00001\%$.
All these show the advantages of our proposed strategies.
It is observed that ESFPM and EIOM can be faster than Algorithm \ref{Alg3}--Algorithm \ref{Alg4}. However, the accuracy of the two algorithms are poorer, which demonstrates the superiority of the high-order PageRank model over the multilinear PageRank model.
Moreover, we see that the MNM and the Pb-Ns algorithms do not work at all for this problem. Indeed, these methods are based on the Newton method, and they suffer from the difficulty of ``out of memory". Thus, they are
not suitable for large-scale problems.

\section{Conclusion}\label{sec7}
 
One of the important applications of high-order Markov
chain is high-order PageRank.
However, one has to store a {\it dense} higher-order PageRank tensor in the higher-order PageRank problem, which is prohibitive for large-scale problems.
Multilinear PageRank is an alternative to higher-order PageRank, however, unlike the higher-order PageRank problem, the existence and uniqueness of multilinear PageRank vector is not guaranteed theoretically, and the approximation of the multilinear PageRank vector to the high-order PageRank tensor can be poor in practice.
Thus, it is recommended to solve the high-order PageRank problem instead of the multilinear PageRank problem, if both the computational overhead and storage requirements can be reduced significantly.

The truncated power method approximates large-scale solutions of sparse Markov chains by
the sparsity component and the rank-one component. However, the accuracy of the approximation and the efficiency of this method may not be satisfactory.
In this work, we aim to accelerate the truncated power method for high-order PageRank. In our proposed methods, there is no need to form and store the vectors arising from the dangling states, and no need to store an auxiliary matrix. Moreover, we propose to speed up the computation by using a partial updating technique. Thus, the computational overhead can be reduced significantly. The convergence of the proposed methods are considered.

There are still some work needs to investigate further. For instance, although we proved that the sparse power method and the sparse power method with partial updating converge with high probability, the convergence of the two methods deserves further investigation. On the other hand, higher-order PageRank problem (say, $m\geq 4$) can be reduced to several third-order PageRank problems theoretically. For example, a fourth-order PageRank problem can be rewritten as $n$ third-order PageRank problems, however, the workload will be terrible if we solve it in such a way. Whether some
randomized algorithms can be applied is worth studying, and it is definitely a part of our future work.

\appendix
\section{Numerical results of Subsection \ref{sec-real} on real-world data set}\label{AppA}
\begin{table}[H]\caption{A comparison of the six algorithms for high-order PageRank on .GOV Web collection, $\alpha=0.85$, and the parameter $\beta$ for sparsing is chosen as $\frac{1}{n^{2}},\frac{1}{n^{3}}$ and $\frac{1}{n^{4}}$, respectively. }\label{t-6}
\centering
\resizebox{\textwidth}{37mm}{
\begin{tabular}{cccccc}
  \toprule
  \multicolumn{2}{c}{Algorithm} & Iter & CPU (s) & Error & Sparsity of $\boldsymbol{S}$ \\
  \midrule
  \multicolumn{2}{c}{The power method} & 7 & 7018.56 & --- & --- \\   
  \midrule
  \multirow{3}{*}{Algorithm \ref{Alg2} \cite{Ding2018Fast}}
  & $\beta = \frac{1}{n^{2}}$ & 4 & 32074.85 & $3.00 \times 10^{-1}$ & $2.94 \times 10^{-7}$ \\
  & $\beta = \frac{1}{n^{3}}$ & 6 & 50889.89 & $3.00 \times 10^{-1}$ & $4.11 \times 10^{-7}$ \\
  & $\beta = \frac{1}{n^{4}}$ & 6 & 50381.11 & $3.00 \times 10^{-1}$ & $4.11 \times 10^{-7}$ \\
  \midrule
  \multirow{3}{*}{Algorithm \ref{Alg3}}
  & $\beta = \frac{1}{n^{2}}$ & 5 & 7219.40 & $4.55 \times 10^{-6}$ & $3.00 \times 10^{-7}$ \\
  & $\beta = \frac{1}{n^{3}}$ & 7 & 12200.94 & $1.47 \times 10^{-6}$ & $4.11 \times 10^{-7}$ \\
  & $\beta = \frac{1}{n^{4}}$ & 7 & 12349.04 & $1.47 \times 10^{-6}$ & $4.11 \times 10^{-7}$ \\
  \midrule
  \multirow{3}{*}{Algorithm \ref{Alg5}}
  & $\beta = \frac{1}{n^{2}}$ & 2 & 2995.40 & $9.29 \times 10^{-6}$ & $2.98 \times 10^{-7}$ \\
  & $\beta = \frac{1}{n^{3}}$ & 3 & 3160.03 & $7.69 \times 10^{-6}$ & $4.11 \times 10^{-7}$ \\
  & $\beta = \frac{1}{n^{4}}$ & 3 & 2767.63 & $7.69 \times 10^{-6}$ & $4.11 \times 10^{-7}$ \\
  \midrule
  \multirow{3}{*}{Algorithm \ref{Alg4}}
  & $\beta = \frac{1}{n^{2}}$ & 6 & 7824.64 & $2.72 \times 10^{-6}$ & $3.00 \times 10^{-7}$ \\
  & $\beta = \frac{1}{n^{3}}$ & 7 & 11879.17 & $9.90 \times 10^{-7}$ & $4.11 \times 10^{-7}$ \\
  & $\beta = \frac{1}{n^{4}}$ & 7 & 13764.39 & $9.90 \times 10^{-7}$ & $4.11 \times 10^{-7}$ \\
  \midrule
  \multirow{3}{*}{Algorithm \ref{Alg6}}
  & $\beta = \frac{1}{n^{2}}$ & 3 & 4172.94 & $8.43 \times 10^{-6}$ & $2.98 \times 10^{-7}$ \\
  & $\beta = \frac{1}{n^{3}}$ & 4 & 4127.38 & $6.57 \times 10^{-6}$ & $4.11 \times 10^{-7}$ \\
  & $\beta = \frac{1}{n^{4}}$ & 4 & 4143.98 & $6.57 \times 10^{-6}$ & $4.11 \times 10^{-7}$ \\
  \bottomrule
\end{tabular}}
\end{table}

\begin{table}[H]\caption{A comparison of the six algorithms for high-order PageRank on .GOV Web collection, $\alpha=0.90$, and the parameter $\beta$ for sparsing is chosen as $\frac{1}{n^{2}},\frac{1}{n^{3}}$ and $\frac{1}{n^{4}}$, respectively. }\label{t-70}
\centering
\resizebox{\textwidth}{37mm}{
\begin{tabular}{cccccc}
  \toprule
  \multicolumn{2}{c}{Algorithm} & Iter & CPU (s) & Error & Sparsity of $\boldsymbol{S}$ \\
  \midrule
  \multicolumn{2}{c}{The power method} & 10 & 8858.75 & --- & --- \\   
  \midrule
  \multirow{3}{*}{Algorithm \ref{Alg2} \cite{Ding2018Fast}}
  & $\beta = \frac{1}{n^{2}}$ & 6 & 41782.85 & $2.00 \times 10^{-1}$ & $2.97 \times 10^{-7}$ \\
  & $\beta = \frac{1}{n^{3}}$ & 9 & 76205.63 & $2.00 \times 10^{-1}$ & $4.11 \times 10^{-7}$ \\
  & $\beta = \frac{1}{n^{4}}$ & 9 & 76309.26 & $2.00 \times 10^{-1}$ & $4.11 \times 10^{-7}$ \\
  \midrule
  \multirow{3}{*}{Algorithm \ref{Alg3}}
  & $\beta = \frac{1}{n^{2}}$ & 8 & 9866.03 & $5.61 \times 10^{-6}$ & $3.00 \times 10^{-7}$ \\
  & $\beta = \frac{1}{n^{3}}$ & 9 & 16576.04 & $1.51 \times 10^{-6}$ & $4.11\times 10^{-7}$ \\
  & $\beta = \frac{1}{n^{4}}$ & 9 & 15220.57 & $1.51 \times 10^{-6}$ & $4.11\times 10^{-7}$ \\
  \midrule
  \multirow{3}{*}{Algorithm \ref{Alg5}}
  & $\beta = \frac{1}{n^{2}}$ & 3 & 3183.62 & $1.37 \times 10^{-5}$ & $2.98 \times 10^{-7}$ \\
  & $\beta = \frac{1}{n^{3}}$ & 4 & 3169.19 & $1.18 \times 10^{-5}$ & $4.11\times 10^{-7}$ \\
  & $\beta = \frac{1}{n^{4}}$ & 4 & 3125.31 & $1.18 \times 10^{-5}$ & $4.11\times 10^{-7}$ \\
  \midrule
  \multirow{3}{*}{Algorithm \ref{Alg4}}
  & $\beta = \frac{1}{n^{2}}$ & 8 & 9280.87 & $4.59 \times 10^{-6}$ & $3.00 \times 10^{-7}$ \\
  & $\beta = \frac{1}{n^{3}}$ & 10 & 18895.82 & $9.84 \times 10^{-7}$ & $4.11\times 10^{-7}$ \\
  & $\beta = \frac{1}{n^{4}}$ & 10 & 19130.43 & $9.84 \times 10^{-7}$ & $4.11\times 10^{-7}$ \\
  \midrule
  \multirow{3}{*}{Algorithm \ref{Alg6}}
  & $\beta = \frac{1}{n^{2}}$ & 3 & 4056.69 & $1.35 \times 10^{-5}$ & $2.98 \times 10^{-7}$ \\
  & $\beta = \frac{1}{n^{3}}$ & 4 & 4135.48 & $1.15\times 10^{-5}$ & $4.11\times 10^{-7}$ \\
  & $\beta = \frac{1}{n^{4}}$ & 4 & 4010.36 & $1.15\times 10^{-5}$ & $4.11\times 10^{-7}$ \\
  \bottomrule
\end{tabular}}
\end{table}

\begin{table}[H]\caption{A comparison of the six algorithms for high-order PageRank on .GOV Web collection, $\alpha=0.95$, and the parameter $\beta$ for sparsing is chosen as $\frac{1}{n^{2}},\frac{1}{n^{3}}$ and $\frac{1}{n^{4}}$, respectively. }\label{t-80}
\centering
\begin{tabular}{cccccc}
  \toprule
  \multicolumn{2}{c}{Algorithm} & Iter & CPU (s) & Error & Sparsity of $\boldsymbol{S}$ \\
  \midrule
  \multicolumn{2}{c}{The power method} & 18 & 15683.21 & --- & --- \\   
  \midrule
  \multirow{3}{*}{Algorithm \ref{Alg2} \cite{Ding2018Fast}}
  & $\beta = \frac{1}{n^{2}}$ & 14 & 84244.66 & $1.00 \times 10^{-1}$ & $3.00 \times 10^{-7}$ \\
  & $\beta = \frac{1}{n^{3}}$ & 17 & 154539.41 & $1.00 \times 10^{-1}$ & $4.11 \times 10^{-7}$  \\
  & $\beta = \frac{1}{n^{4}}$ & 17 & 156132.52 & $1.00 \times 10^{-1}$ & $4.11 \times 10^{-7}$ \\
  \midrule
  \multirow{3}{*}{Algorithm \ref{Alg3}}
  & $\beta = \frac{1}{n^{2}}$ & 15 & 15895.62 & $8.86 \times 10^{-6}$ & $3.00 \times 10^{-7}$ \\
  & $\beta = \frac{1}{n^{3}}$ & 18 & 26673.79 & $1.45 \times 10^{-6}$ & $4.11 \times 10^{-7}$ \\
  & $\beta = \frac{1}{n^{4}}$ & 18 & 31109.39 & $1.45 \times 10^{-6}$ & $4.11 \times 10^{-7}$ \\
  \midrule
  \multirow{3}{*}{Algorithm \ref{Alg5}}
  & $\beta = \frac{1}{n^{2}}$ & 3 & 3070.83 & $2.74 \times 10^{-5}$ & $2.98 \times 10^{-7}$ \\
  & $\beta = \frac{1}{n^{3}}$ & 4 & 3197.02 & $2.52 \times 10^{-5}$ & $4.11 \times 10^{-7}$ \\
  & $\beta = \frac{1}{n^{4}}$ & 4 & 3160.28 & $2.52 \times 10^{-5}$ & $4.11 \times 10^{-7}$ \\
  \midrule
  \multirow{3}{*}{Algorithm \ref{Alg4}}
  & $\beta = \frac{1}{n^{2}}$ & 16 & 17089.03 & $6.92 \times 10^{-6}$ & $3.00 \times 10^{-7}$ \\
  & $\beta = \frac{1}{n^{3}}$ & 18 & 33024.91 & $9.97 \times 10^{-7}$ & $4.11 \times 10^{-7}$ \\
  & $\beta = \frac{1}{n^{4}}$ & 18 & 33760.48 & $9.98 \times 10^{-7}$ & $4.11 \times 10^{-7}$ \\
  \midrule
  \multirow{3}{*}{Algorithm \ref{Alg6}}
  & $\beta = \frac{1}{n^{2}}$ & 3 & 3164.51 & $2.73 \times 10^{-5}$ & $2.98 \times 10^{-7}$ \\
  & $\beta = \frac{1}{n^{3}}$ & 5 & 3487.51 & $2.44 \times 10^{-5}$ & $4.11 \times 10^{-7}$ \\
  & $\beta = \frac{1}{n^{4}}$ & 5 & 3159.92 & $2.44 \times 10^{-5}$ & $4.11 \times 10^{-7}$ \\
  \bottomrule
\end{tabular}
\end{table}

\begin{table}[H]\caption{A comparison of the six algorithms for high-order PageRank on .GOV Web collection, $\alpha=0.99$, and the parameter $\beta$ for sparsing is chosen as $\frac{1}{n^{2}},\frac{1}{n^{3}}$ and $\frac{1}{n^{4}}$, respectively. }\label{t-90}
\centering
\begin{tabular}{cccccc}
  \toprule
  \multicolumn{2}{c}{Algorithm} & Iter & CPU (s) & Error & Sparsity of $\boldsymbol{S}$ \\
  \midrule
  \multicolumn{2}{c}{The power method} & --- & --- & --- & --- \\   
  \midrule
  \multirow{3}{*}{Algorithm \ref{Alg2} \cite{Ding2018Fast}}
  & $\beta = \frac{1}{n^{2}}$ & --- & --- & --- & --- \\
  & $\beta = \frac{1}{n^{3}}$ & --- & --- & --- & --- \\
  & $\beta = \frac{1}{n^{4}}$ & --- & --- & --- & --- \\
  \midrule
  \multirow{3}{*}{Algorithm \ref{Alg3}}
  & $\beta = \frac{1}{n^{2}}$ & --- & --- & --- & --- \\
  & $\beta = \frac{1}{n^{3}}$ & --- & --- & --- & --- \\
  & $\beta = \frac{1}{n^{4}}$ & --- & --- & --- & --- \\
  \midrule
  \multirow{3}{*}{Algorithm \ref{Alg5}}
  & $\beta = \frac{1}{n^{2}}$ & 3 & 3070.83 & --- & $2.98 \times 10^{-7}$ \\
  & $\beta = \frac{1}{n^{3}}$ & 4 & 3197.02 & --- & $4.11 \times 10^{-7}$ \\
  & $\beta = \frac{1}{n^{4}}$ & 4 & 3160.28 & --- & $4.11 \times 10^{-7}$ \\
  \midrule
  \multirow{3}{*}{Algorithm \ref{Alg4}}
  & $\beta = \frac{1}{n^{2}}$ & --- & --- & --- & --- \\
  & $\beta = \frac{1}{n^{3}}$ & --- & --- & --- & --- \\
  & $\beta = \frac{1}{n^{4}}$ & --- & ---& --- & --- \\
  \midrule
  \multirow{3}{*}{Algorithm \ref{Alg6}}
  & $\beta = \frac{1}{n^{2}}$ & 3 & 3164.51 & --- & $2.98 \times 10^{-7}$ \\
  & $\beta = \frac{1}{n^{3}}$ & 5 & 3487.51 & --- & $4.11 \times 10^{-7}$ \\
  & $\beta = \frac{1}{n^{4}}$ & 5 & 3159.92 & --- & $4.11 \times 10^{-7}$ \\
  \bottomrule
\end{tabular}
\end{table}

\begin{table}[H]\caption{The top-ten ranking results of the .GOV Web collection obtained from the five algorithms, $\alpha=0.85$.}\label{t-7}
\begin{tabular}{cccccccccccc}
  \toprule
  \multicolumn{2}{c}{\scriptsize{Algorithm}} & \scriptsize{1} & \scriptsize{2} & \scriptsize{3} & \scriptsize{4} & \scriptsize{5} & \scriptsize{6} & \scriptsize{7} & \scriptsize{8} & \scriptsize{9} & \scriptsize{10} \\
  \midrule
  \multicolumn{2}{c}{\scriptsize{The power method}} & \tiny{18281} & \tiny{9063} & \tiny{26369} & \tiny{54330} & \tiny{35984} & \tiny{56276} & \tiny{36701} & \tiny{54688} & \tiny{28454} & \tiny{41219} \\
  \midrule
  \multirow{3}{*}{\scriptsize{Algorithm \ref{Alg2} \cite{Ding2018Fast}}}
  & \footnotesize{$\beta = \frac{1}{n^{2}}$} & \tiny{81473} & \tiny{9063} & \tiny{18281} & \tiny{90580} & \tiny{35984} & \tiny{56276} & \tiny{28454} & \tiny{41219} & \tiny{26369} & \tiny{54330} \\
  & \tiny{$\beta = \frac{1}{n^{3}}$} & \tiny{81473} & \tiny{9063} & \tiny{18281} & \tiny{90580} & \tiny{56276} & \tiny{35984} & \tiny{28454} & \tiny{41219} & \tiny{26369} & \tiny{54330} \\
  & \tiny{$\beta = \frac{1}{n^{4}}$} & \tiny{81473} & \tiny{9063} & \tiny{18281} & \tiny{90580} & \tiny{56276} & \tiny{35984} & \tiny{28454} & \tiny{41219} & \tiny{26369} & \tiny{54330} \\
  \midrule
  \multirow{3}{*}{\scriptsize{Algorithm \ref{Alg3}}}
  & \tiny{$\beta = \frac{1}{n^{2}}$} & \tiny{9063} & \tiny{18281} & \tiny{28454} & \tiny{41219} & \tiny{35984} & \tiny{56276} & \tiny{40593} & \tiny{29493} & \tiny{26369} & \tiny{54330} \\
  & \tiny{$\beta = \frac{1}{n^{3}}$} & \tiny{9063} & \tiny{18281} & \tiny{28454} & \tiny{41219} & \tiny{35984} & \tiny{56276} & \tiny{40593} & \tiny{26369} & \tiny{54330} & \tiny{29493} \\
  & \tiny{$\beta = \frac{1}{n^{4}}$} & \tiny{9063} & \tiny{18281} & \tiny{28454} & \tiny{41219} & \tiny{35984} & \tiny{56276} & \tiny{40593} & \tiny{26369} & \tiny{54330} & \tiny{29493} \\
  \midrule
  \multirow{3}{*}{\scriptsize{Algorithm \ref{Alg5}}}
  & \tiny{$\beta = \frac{1}{n^{2}}$} & \tiny{9063} & \tiny{18281} & \tiny{28056} & \tiny{56276} & \tiny{35984} & \tiny{28454} & \tiny{41219} & \tiny{26369} & \tiny{54330} & \tiny{40593} \\
  & \tiny{$\beta = \frac{1}{n^{3}}$} & \tiny{9063} & \tiny{18281} & \tiny{28454} & \tiny{41219} & \tiny{56276} & \tiny{35984} & \tiny{28056} & \tiny{40593} & \tiny{29493} & \tiny{55752} \\
  & \tiny{$\beta = \frac{1}{n^{4}}$} & \tiny{9063} & \tiny{18281} & \tiny{28454} & \tiny{41219} & \tiny{56276} & \tiny{35984} & \tiny{28056} & \tiny{40593} & \tiny{29493} & \tiny{55752} \\
  \midrule
  \multirow{3}{*}{\scriptsize{Algorithm \ref{Alg4}}}
  & \tiny{$\beta = \frac{1}{n^{2}}$} & \tiny{9063} & \tiny{18281} & \tiny{35984} & \tiny{56276} & \tiny{28454} & \tiny{41219} & \tiny{26369} & \tiny{54330} & \tiny{40593} & \tiny{29493} \\
  & \tiny{$\beta = \frac{1}{n^{3}}$} & \tiny{9063} & \tiny{18281} & \tiny{28454} & \tiny{41219} & \tiny{56276} & \tiny{35984} & \tiny{40593} & \tiny{26369} & \tiny{54330} & \tiny{29493} \\
  & \tiny{$\beta = \frac{1}{n^{4}}$} & \tiny{9063} & \tiny{18281} & \tiny{28454} & \tiny{41219} & \tiny{35984} & \tiny{56276} & \tiny{40593} & \tiny{26369} & \tiny{54330} & \tiny{29493} \\
  \midrule
  \multirow{3}{*}{\scriptsize{Algorithm \ref{Alg6}}}
  & \tiny{$\beta = \frac{1}{n^{2}}$} & \tiny{9063} & \tiny{18281} & \tiny{28454} & \tiny{41219} & \tiny{56276} & \tiny{35984} & \tiny{28056} & \tiny{40593} & \tiny{29493} & \tiny{26369} \\
  & \tiny{$\beta = \frac{1}{n^{3}}$} & \tiny{9063} & \tiny{18281} & \tiny{56276} & \tiny{35984} & \tiny{28454} & \tiny{41219} & \tiny{26369} & \tiny{54330} & \tiny{28056} & \tiny{40593} \\
  & \tiny{$\beta = \frac{1}{n^{4}}$} & \tiny{9063} & \tiny{18281} & \tiny{56276} & \tiny{35984} & \tiny{28454} & \tiny{41219} & \tiny{26369} & \tiny{54330} & \tiny{28056} & \tiny{40593} \\
  \bottomrule
\end{tabular}
\end{table}

\newpage
\section{Numerical results of Subsection \ref{AD} on synthetic data set}\label{AppB}
\begin{table}[H]\caption{A comparison of the algorithms on synthetic data, $n=10000,\alpha = 0.85$}\label{t-2}
\centering
{\begin{tabular}{ccccccc}
  \toprule
  Sparsity of $\mathbf{Q}$ & \multicolumn{2}{c}{Algorithm} & Iter & CPU (s) & Error & Sparsity of $\boldsymbol{S}$  \\
  \midrule
  \multirow{20}{*}{0.0001\%}
  & \multicolumn{2}{c}{The power method} & 9 & 22.75 & --- & --- \\
  & \multicolumn{2}{c}{MNM \cite{Guo2018A}} & \multicolumn{4}{c}{Out of memory}\\
  & \multicolumn{2}{c}{Pb-Ns \cite{Meini2018Perron}} & \multicolumn{4}{c}{Out of memory}\\
  & \multicolumn{2}{c}{ESFPM \cite{Cipolla2019Extrapolation}} & 9 & 15.11 & $1.67 \times 10^{-2}$ & ---\\
  & \multicolumn{2}{c}{EIOM \cite{Cipolla2019Extrapolation}} & 7 & 11.55 & $1.67 \times 10^{-2}$ & ---\\
  \cline{2-7}
  & \multirow{3}{*}{Algorithm \ref{Alg2} \cite{Ding2018Fast}}
  & $\beta = \frac{1}{n^{2}}$ & 6 & 42.61 & $2.99 \times 10^{-1}$ & $1.50 \times 10^{-4}$ \\
  & & $\beta = \frac{1}{n^{3}}$ & 7 & 159.75 & $3.00 \times 10^{-1}$ & $1.00 \times 10^{-2}$ \\
  & & $\beta = \frac{1}{n^{4}}$ & 7 & 156.00 & $3.00 \times 10^{-1}$ & $1.00 \times 10^{-2}$ \\
  \cline{2-7}
  & \multirow{3}{*}{Algorithm \ref{Alg3}}
  & $\beta = \frac{1}{n^{2}}$ & 4 & 17.28 & $1.67 \times 10^{-2}$ & $4.96 \times 10^{-5}$ \\
  & & $\beta = \frac{1}{n^{3}}$ & 6 & 125.00 & $1.02 \times 10^{-4}$ & $1.00 \times 10^{-2}$ \\
  & & $\beta = \frac{1}{n^{4}}$ & 6 & 119.15 & $1.01 \times 10^{-4}$ & $1.00 \times 10^{-2}$ \\
    \cline{2-7}
  & \multirow{3}{*}{Algorithm \ref{Alg5}}
  & $\beta = \frac{1}{n^{2}}$ & 3 & 3.90 & $1.67 \times 10^{-2}$ & $4.93 \times 10^{-5}$ \\
  & & $\beta = \frac{1}{n^{3}}$ & 4 & 58.81 & $7.35 \times 10^{-4}$ & $1.00 \times 10^{-2}$ \\
  & & $\beta = \frac{1}{n^{4}}$ & 4 & 61.61 & $7.35 \times 10^{-4}$ & $1.00 \times 10^{-2}$ \\
  \cline{2-7}
  & \multirow{3}{*}{Algorithm \ref{Alg4}}
  & $\beta = \frac{1}{n^{2}}$ & 4 & 17.12 & $1.67 \times 10^{-2}$ & $4.96 \times 10^{-5}$ \\
  & & $\beta = \frac{1}{n^{3}}$ & 6 & 118.35 & $1.99 \times 10^{-6}$ & $1.00 \times 10^{-2}$ \\
  & & $\beta = \frac{1}{n^{4}}$ & 6 & 115.37 & $4.38 \times 10^{-10}$ & $1.00 \times 10^{-2}$ \\
  \cline{2-7}
  & \multirow{3}{*}{Algorithm \ref{Alg6}}
  & $\beta = \frac{1}{n^{2}}$ & 3 & 3.74 & $1.67 \times 10^{-2}$ & $4.93 \times 10^{-5}$ \\
  & & $\beta = \frac{1}{n^{3}}$ & 4 & 55.09 & $7.28 \times 10^{-4}$ & $1.00 \times 10^{-2}$ \\
  & & $\beta = \frac{1}{n^{4}}$ & 4 & 64.07 & $7.29 \times 10^{-4}$ & $1.00 \times 10^{-2}$ \\
  \bottomrule
  \multirow{20}{*}{0.00001\%}
  & \multicolumn{2}{c}{A power method} & 7 & 19.89 & --- & ---\\
  & \multicolumn{2}{c}{MNM \cite{Guo2018A}} & \multicolumn{4}{c}{Out of memory} \\
  & \multicolumn{2}{c}{Pb-N \cite{Meini2018Perron}} & \multicolumn{4}{c}{Out of memory} \\
  & \multicolumn{2}{c}{ESFPM \cite{Cipolla2019Extrapolation}} & 8 & 12.90 & $1.70 \times 10^{-3}$ & ---\\
  & \multicolumn{2}{c}{EIOM \cite{Cipolla2019Extrapolation}} & 6 & 9.76 & $1.70 \times 10^{-3}$ & ---\\
  \cline{2-7}
  & \multirow{3}{*}{Algorithm \ref{Alg2} \cite{Ding2018Fast}}
  & $\beta = \frac{1}{n^{2}}$ & 4 & 24.53 & $3.00 \times 10^{-1}$ & $1.44 \times 10^{-6}$\\
  & & $\beta = \frac{1}{n^{3}}$ & 6 & 77.66 & $3.00 \times 10^{-1}$ & $9.99 \times 10^{-4}$ \\
  & & $\beta = \frac{1}{n^{4}}$ & 6 & 77.09 & $3.00 \times 10^{-1}$ & $9.99 \times 10^{-4}$\\
  \cline{2-7}
  & \multirow{3}{*}{Algorithm \ref{Alg3}}
  & $\beta = \frac{1}{n^{2}}$ & 3 & 11.32 & $1.70 \times 10^{-3}$ & $5.50 \times 10^{-7}$\\
  & & $\beta = \frac{1}{n^{3}}$ & 5 & 50.82 & $3.19 \times 10^{-5}$ & $9.99 \times 10^{-4}$ \\
  & & $\beta = \frac{1}{n^{4}}$ & 5 & 50.22 & $3.18 \times 10^{-5}$ & $9.99 \times 10^{-4}$\\
  \cline{2-7}
  & \multirow{3}{*}{Algorithm \ref{Alg5}}
  & $\beta = \frac{1}{n^{2}}$ & 3 & 4.28 & $1.70 \times 10^{-3}$ & $5.50 \times 10^{-7}$\\
  & & $\beta = \frac{1}{n^{3}}$ & 3 & 10.18 & $1.94 \times 10^{-4}$ & $9.99 \times 10^{-4}$ \\
  & & $\beta = \frac{1}{n^{4}}$ & 3 & 1.58 & $1.95 \times 10^{-4}$ & $9.99 \times 10^{-4}$\\
  \cline{2-7}
  & \multirow{3}{*}{Algorithm \ref{Alg4}}
  & $\beta = \frac{1}{n^{2}}$ & 3 & 11.16 & $1.70 \times 10^{-3}$ & $5.50 \times 10^{-7}$\\
  & & $\beta = \frac{1}{n^{3}}$ & 5 & 49.98 & $2.00 \times 10^{-7}$ & $9.99 \times 10^{-4}$ \\
  & & $\beta = \frac{1}{n^{4}}$ & 5 & 50.82 & $2.83 \times 10^{-11}$ & $9.99 \times 10^{-4}$\\
  \cline{2-7}
  & \multirow{3}{*}{Algorithm \ref{Alg6}}
  & $\beta = \frac{1}{n^{2}}$ & 3 & 4.49 & $1.70 \times 10^{-3}$ & $5.50 \times 10^{-7}$\\
  & & $\beta = \frac{1}{n^{3}}$ & 3 & 10.17 & $1.91 \times 10^{-4}$ & $9.99 \times 10^{-4}$\\
  & & $\beta = \frac{1}{n^{4}}$ & 3 & 9.99 & $1.91 \times 10^{-4}$ & $9.99 \times 10^{-4}$\\
  \bottomrule
\end{tabular}}
\end{table}

\begin{table}[H]\caption{A comparison of the algorithms on synthetic data, $n=10000,\alpha = 0.90$}\label{t-3}
\centering
\begin{tabular}{ccccccc}
  \toprule
  Sparsity of $\mathbf{Q}$ & \multicolumn{2}{c}{Algorithm} & Iter & CPU (s) & Error & Sparsity of $\boldsymbol{S}$  \\
  \midrule
  \multirow{20}{*}{0.0001\%}
  & \multicolumn{2}{c}{The power method} & 10 & 27.72 & --- & --- \\
  & \multicolumn{2}{c}{MNM \cite{Guo2018A}} & \multicolumn{4}{c}{Out of memory}\\
  & \multicolumn{2}{c}{Pb-Ns \cite{Meini2018Perron}} & \multicolumn{4}{c}{Out of memory}\\
  & \multicolumn{2}{c}{ESFPM \cite{Cipolla2019Extrapolation}} & 9 & 16.29 & $1.77 \times 10^{-2}$ & ---\\
  & \multicolumn{2}{c}{EIOM \cite{Cipolla2019Extrapolation}} & 8 & 14.26 & $1.77 \times 10^{-2}$ & ---\\
  \cline{2-7}
  & \multirow{3}{*}{Algorithm \ref{Alg2} \cite{Ding2018Fast}}
  & $\beta = \frac{1}{n^{2}}$ & 6 & 45.40 & $2.00 \times 10^{-1}$ & $1.50 \times 10^{-4}$ \\
  & & $\beta = \frac{1}{n^{3}}$ & 7 & 148.91 & $2.00 \times 10^{-1}$ & $1.00 \times 10^{-2}$ \\
  & & $\beta = \frac{1}{n^{4}}$ & 7 & 154.50 & $2.00 \times 10^{-1}$ & $1.00 \times 10^{-2}$ \\
  \cline{2-7}
  & \multirow{3}{*}{Algorithm \ref{Alg3}}
  & $\beta = \frac{1}{n^{2}}$ & 4 & 18.02 & $1.77 \times 10^{-2}$ & $4.97 \times 10^{-5}$\\
  & & $\beta = \frac{1}{n^{3}}$ & 6 & 127.21 & $7.20 \times 10^{-5}$ & $1.00 \times 10^{-2}$\\
  & & $\beta = \frac{1}{n^{4}}$ & 6 & 122.65 & $7.16 \times 10^{-5}$ & $1.00 \times 10^{-2}$\\
  \cline{2-7}
  & \multirow{3}{*}{Algorithm \ref{Alg5}}
  & $\beta = \frac{1}{n^{2}}$ & 3 & 4.43 & $1.77 \times 10^{-2}$ & $4.93 \times 10^{-5}$\\
  & & $\beta = \frac{1}{n^{3}}$ & 4 & 52.33 & $7.83 \times 10^{-4}$ & $1.00 \times 10^{-2}$ \\
  & & $\beta = \frac{1}{n^{4}}$ & 4 & 66.84 & $7.84 \times 10^{-4}$ & $1.00 \times 10^{-2}$\\
  \cline{2-7}
  & \multirow{3}{*}{Algorithm \ref{Alg4}}
  & $\beta = \frac{1}{n^{2}}$ & 5 & 22.32 & $1.77 \times 10^{-2}$ & $4.97 \times 10^{-5}$\\
  & & $\beta = \frac{1}{n^{3}}$ & 6 & 121.06 & $1.99 \times 10^{-6}$ & $1.00 \times 10^{-2}$ \\
  & & $\beta = \frac{1}{n^{4}}$ & 6 & 125.97& $5.25 \times 10^{-10}$ & $1.00 \times 10^{-2}$\\
  \cline{2-7}
  & \multirow{3}{*}{Algorithm \ref{Alg6}}
  & $\beta = \frac{1}{n^{2}}$ & 3 & 4.34 & $1.77 \times 10^{-2}$ & $4.93 \times 10^{-5}$\\
  & & $\beta = \frac{1}{n^{3}}$ & 4 & 59.66 & $7.80 \times 10^{-4}$ & $1.00 \times 10^{-2}$ \\
  & & $\beta = \frac{1}{n^{4}}$ & 4 & 57.09 & $7.81 \times 10^{-4}$ & $1.00 \times 10^{-2}$\\
  \bottomrule
  \multirow{20}{*}{0.00001\%}
  & \multicolumn{2}{c}{The power method} & 7 & 20.18 & --- & ---\\
  & \multicolumn{2}{c}{MNM \cite{Guo2018A}} & \multicolumn{4}{c}{Out of memory} \\
  & \multicolumn{2}{c}{Pb-N \cite{Meini2018Perron}} & \multicolumn{4}{c}{Out of memory} \\
  & \multicolumn{2}{c}{ESFPM \cite{Cipolla2019Extrapolation}} & 8 & 13.44 & $1.80 \times 10^{-3}$ & ---\\
  & \multicolumn{2}{c}{EIOM \cite{Cipolla2019Extrapolation}} & 7 & 11.87 & $1.80 \times 10^{-3}$ & ---\\
  \cline{2-7}
  & \multirow{3}{*}{Algorithm \ref{Alg2} \cite{Ding2018Fast}}
  & $\beta = \frac{1}{n^{2}}$ & 4 & 24.64 & $2.00 \times 10^{-1}$ & $1.44 \times 10^{-6}$\\
  & & $\beta = \frac{1}{n^{3}}$ & 6 & 76.65 & $2.00 \times 10^{-1}$ & $9.99 \times 10^{-4}$ \\
  & & $\beta = \frac{1}{n^{4}}$ & 6 & 73.86 & $2.00 \times 10^{-1}$ & $9.99 \times 10^{-4}$\\
  \cline{2-7}
  & \multirow{3}{*}{Algorithm \ref{Alg3}}
  & $\beta = \frac{1}{n^{2}}$ & 3 & 12.03 & $1.80 \times 10^{-3}$ & $5.50 \times 10^{-7}$\\
  & & $\beta = \frac{1}{n^{3}}$ & 5 & 51.16 & $2.26 \times 10^{-5}$ & $9.99 \times 10^{-4}$ \\
  & & $\beta = \frac{1}{n^{4}}$ & 5 & 50.28 & $2.25 \times 10^{-5}$ & $9.99 \times 10^{-4}$\\
  \cline{2-7}
  & \multirow{3}{*}{Algorithm \ref{Alg5}}
  & $\beta = \frac{1}{n^{2}}$ & 3 & 4.31 & $1.80 \times 10^{-3}$ & $5.50 \times 10^{-7}$\\
  & & $\beta = \frac{1}{n^{3}}$ & 3 & 10.52 & $2.05 \times 10^{-4}$ & $9.99 \times 10^{-4}$ \\
  & & $\beta = \frac{1}{n^{4}}$ & 3 & 10.42 & $2.05 \times 10^{-4}$ & $9.99 \times 10^{-4}$\\
  \cline{2-7}
  & \multirow{3}{*}{Algorithm \ref{Alg4}}
  & $\beta = \frac{1}{n^{2}}$ & 3 & 11.50 & $1.80 \times 10^{-3}$ & $5.50 \times 10^{-7}$\\
  & & $\beta = \frac{1}{n^{3}}$ & 5 & 52.27 & $2.00 \times 10^{-7}$ & $9.99 \times 10^{-4}$ \\
  & & $\beta = \frac{1}{n^{4}}$ & 5 & 53.55 & $3.09 \times 10^{-11}$ & $9.99 \times 10^{-4}$\\
  \cline{2-7}
  & \multirow{3}{*}{Algorithm \ref{Alg6}}
  & $\beta = \frac{1}{n^{2}}$ & 3 & 4.39 & $1.80 \times 10^{-3}$ & $5.50 \times 10^{-7}$\\
  & & $\beta = \frac{1}{n^{3}}$ & 3 & 10.73 & $2.03 \times 10^{-4}$ & $9.99 \times 10^{-4}$\\
  & & $\beta = \frac{1}{n^{4}}$ & 3 & 10.22 & $2.03 \times 10^{-4}$ & $9.99 \times 10^{-4}$\\
  \bottomrule
\end{tabular}
\end{table}

\begin{table}[H]\caption{A comparison of the algorithms on synthetic data, $n=10000,\alpha = 0.95$}\label{t-4}
\centering
\begin{tabular}{ccccccc}
  \toprule
  Sparsity of $\mathbf{Q}$ & \multicolumn{2}{c}{Algorithm} & Iter & CPU (s) & Error & Sparsity of $\boldsymbol{S}$  \\
  \midrule
  \multirow{20}{*}{0.0001\%}
  & \multicolumn{2}{c}{The power method} & 10 & 30.45 & --- & --- \\
  & \multicolumn{2}{c}{MNM \cite{Guo2018A}} & \multicolumn{4}{c}{Out of memory}\\
  & \multicolumn{2}{c}{Pb-Ns \cite{Meini2018Perron}} & \multicolumn{4}{c}{Out of memory}\\
  & \multicolumn{2}{c}{ESFPM \cite{Cipolla2019Extrapolation}} & 9 & 15.73 & $1.87 \times 10^{-2}$ & ---\\
  & \multicolumn{2}{c}{EIOM \cite{Cipolla2019Extrapolation}} & 8 & 14.05 & $1.87 \times 10^{-2}$ & ---\\
  \cline{2-7}
  & \multirow{3}{*}{Algorithm \ref{Alg2} \cite{Ding2018Fast}}
  & $\beta = \frac{1}{n^{2}}$ & 6 & 45.50 & $9.98 \times 10^{-2}$ & $1.50 \times 10^{-4}$ \\
  & & $\beta = \frac{1}{n^{3}}$ & 7 & 146.32 & $1.00 \times 10^{-1}$ & $1.00 \times 10^{-2}$ \\
  & & $\beta = \frac{1}{n^{4}}$ & 7 & 149.41 & $1.00 \times 10^{-1}$ & $1.00 \times 10^{-2}$ \\
  \cline{2-7}
  & \multirow{3}{*}{Algorithm \ref{Alg3}}
  & $\beta = \frac{1}{n^{2}}$ & 5 & 24.24 & $1.86 \times 10^{-2}$ & $4.97 \times 10^{-5}$ \\
  & & $\beta = \frac{1}{n^{3}}$ & 7 & 134.60 & $3.84 \times 10^{-5}$ & $1.00 \times 10^{-2}$ \\
  & & $\beta = \frac{1}{n^{4}}$ & 7 & 130.80 & $3.78 \times 10^{-5}$ & $1.00 \times 10^{-2}$ \\
  \cline{2-7}
  & \multirow{3}{*}{Algorithm \ref{Alg5}}
  & $\beta = \frac{1}{n^{2}}$ & 3 & 4.93& $1.86 \times 10^{-2}$ & $4.93 \times 10^{-5}$ \\
  & & $\beta = \frac{1}{n^{3}}$ & 4 & 66.91 & $8.33 \times 10^{-4}$ & $1.00 \times 10^{-2}$ \\
  & & $\beta = \frac{1}{n^{4}}$ & 4 & 53.78 & $8.34 \times 10^{-4}$ & $1.00 \times 10^{-2}$ \\
  \cline{2-7}
  & \multirow{3}{*}{Algorithm \ref{Alg4}}
  & $\beta = \frac{1}{n^{2}}$ & 5 & 23.47 & $1.86 \times 10^{-2}$ & $4.97 \times 10^{-5}$ \\
  & & $\beta = \frac{1}{n^{3}}$ & 7 & 139.64 & $1.99 \times 10^{-6}$ & $1.00 \times 10^{-2}$ \\
  & & $\beta = \frac{1}{n^{4}}$ & 7 & 129.28 & $2.02 \times 10^{-10}$ & $1.00 \times 10^{-2}$ \\
  \cline{2-7}
  & \multirow{3}{*}{Algorithm \ref{Alg6}}
  & $\beta = \frac{1}{n^{2}}$ & 3 & 4.70 & $1.86 \times 10^{-2}$ & $4.93 \times 10^{-5}$ \\
  & & $\beta = \frac{1}{n^{3}}$ & 4 & 56.84 & $8.32 \times 10^{-4}$ & $1.00 \times 10^{-2}$ \\
  & & $\beta = \frac{1}{n^{4}}$ & 4 & 66.45 & $8.33 \times 10^{-4}$ & $1.00 \times 10^{-2}$ \\
  \bottomrule
  \multirow{20}{*}{0.00001\%}
  & \multicolumn{2}{c}{The power method} & 7 & 20.41 & --- & ---\\
  & \multicolumn{2}{c}{MNM \cite{Guo2018A}} & \multicolumn{4}{c}{Out of memory} \\
  & \multicolumn{2}{c}{Pb-N \cite{Meini2018Perron}} & \multicolumn{4}{c}{Out of memory} \\
  & \multicolumn{2}{c}{ESFPM \cite{Cipolla2019Extrapolation}} & 8 & 13.70 & $1.90 \times 10^{-3}$ & ---\\
  & \multicolumn{2}{c}{EIOM \cite{Cipolla2019Extrapolation}} & 7 & 11.78 & $1.90 \times 10^{-3}$ & ---\\
  \cline{2-7}
  & \multirow{3}{*}{Algorithm \ref{Alg2} \cite{Ding2018Fast}}
  & $\beta = \frac{1}{n^{2}}$ & 4 & 26.76 & $1.00 \times 10^{-1}$ & $1.44 \times 10^{-6}$\\
  & & $\beta = \frac{1}{n^{3}}$ & 6 & 81.82 & $1.00 \times 10^{-1}$ & $9.99 \times 10^{-4}$ \\
  & & $\beta = \frac{1}{n^{4}}$ & 6 & 82.12 & $1.00 \times 10^{-1}$ & $9.99 \times 10^{-4}$\\
  \cline{2-7}
  & \multirow{3}{*}{Algorithm \ref{Alg3}}
  & $\beta = \frac{1}{n^{2}}$ & 3 & 12.90 & $1.90 \times 10^{-3}$ & $5.50 \times 10^{-7}$\\
  & & $\beta = \frac{1}{n^{3}}$ & 5 & 53.35 & $1.20 \times 10^{-5}$ & $9.99 \times 10^{-4}$ \\
  & & $\beta = \frac{1}{n^{4}}$ & 5 & 54.03 & $1.19 \times 10^{-5}$ & $9.99 \times 10^{-4}$\\
  \cline{2-7}
  & \multirow{3}{*}{Algorithm \ref{Alg5}}
  & $\beta = \frac{1}{n^{2}}$ & 3 & 4.43 & $1.90 \times 10^{-3}$ & $5.50 \times 10^{-7}$\\
  & & $\beta = \frac{1}{n^{3}}$ & 3 & 10.82 & $2.15 \times 10^{-4}$ & $9.99 \times 10^{-4}$ \\
  & & $\beta = \frac{1}{n^{4}}$ & 3 & 10.64 & $2.15 \times 10^{-4}$ & $9.99 \times 10^{-4}$\\
  \cline{2-7}
  & \multirow{3}{*}{Algorithm \ref{Alg4}}
  & $\beta = \frac{1}{n^{2}}$ & 3 & 12.21 & $1.90 \times 10^{-3}$ & $5.50 \times 10^{-7}$\\
  & & $\beta = \frac{1}{n^{3}}$ & 5 & 54.28 & $2.00 \times 10^{-7}$ & $9.99 \times 10^{-4}$ \\
  & & $\beta = \frac{1}{n^{4}}$ & 5 & 54.31 & $3.40 \times 10^{-11}$ & $9.99 \times 10^{-4}$\\
  \cline{2-7}
  & \multirow{3}{*}{Algorithm \ref{Alg6}}
  & $\beta = \frac{1}{n^{2}}$ & 3 & 4.54 & $1.90 \times 10^{-3}$ & $5.50 \times 10^{-7}$\\
  & & $\beta = \frac{1}{n^{3}}$ & 3 & 10.82 & $2.14 \times 10^{-4}$ & $9.99 \times 10^{-4}$\\
  & & $\beta = \frac{1}{n^{4}}$ & 3 & 10.77 & $2.14\times 10^{-4}$ & $9.99 \times 10^{-4}$\\
  \bottomrule
\end{tabular}
\end{table}

\begin{table}[H]\caption{A comparison of the algorithms on synthetic data, $n=10000,\alpha = 0.99$}\label{t-5}
\centering
\begin{tabular}{ccccccc}
  \toprule
  Sparsity of $\mathbf{Q}$ & \multicolumn{2}{c}{Algorithm} & Iter & CPU (s) & Error & Sparsity of $\boldsymbol{S}$  \\
  \midrule
  \multirow{20}{*}{0.0001\%}
  & \multicolumn{2}{c}{The power method} & 10 & 29.59 & --- & --- \\
  & \multicolumn{2}{c}{MNM \cite{Guo2018A}} & \multicolumn{4}{c}{Out of memory}\\
  & \multicolumn{2}{c}{Pb-Ns \cite{Meini2018Perron}} & \multicolumn{4}{c}{Out of memory}\\
  & \multicolumn{2}{c}{ESFPM \cite{Cipolla2019Extrapolation}} & 9 & 15.69 & $1.95 \times 10^{-2}$ & ---\\
  & \multicolumn{2}{c}{EIOM \cite{Cipolla2019Extrapolation}} & 9 & 15.58 & $1.95 \times 10^{-2}$ & ---\\
  \cline{2-7}
  & \multirow{3}{*}{Algorithm \ref{Alg2} \cite{Ding2018Fast}}
  & $\beta = \frac{1}{n^{2}}$ & 6 & 47.44 & $2.07 \times 10^{-2}$ & $ 1.50 \times 10^{-4}$ \\
  & & $\beta = \frac{1}{n^{3}}$ & 7 & 152.43 & $2.00 \times 10^{-2}$ & $1.00 \times 10^{-2}$ \\
  & & $\beta = \frac{1}{n^{4}}$ & 7 & 154.12 & $2.00 \times 10^{-2}$ & $1.00 \times 10^{-2}$ \\
  \cline{2-7}
  & \multirow{3}{*}{Algorithm \ref{Alg3}}
  & $\beta = \frac{1}{n^{2}}$ & 5 & 24.06 & $1.94 \times 10^{-2}$ & $4.97 \times 10^{-2}$ \\
  & & $\beta = \frac{1}{n^{3}}$ & 7 & 142.00 & $8.91 \times 10^{-6}$ & $1.00 \times 10^{-2}$ \\
  & & $\beta = \frac{1}{n^{4}}$ & 7 & 133.00 & $7.89 \times 10^{-6}$ & $1.00 \times 10^{-2}$ \\
  \cline{2-7}
  & \multirow{3}{*}{Algorithm \ref{Alg5}}
  & $\beta = \frac{1}{n^{2}}$ & 3 & 4.53 & $1.94 \times 10^{-2}$ & $4.93 \times 10^{-5}$ \\
  & & $\beta = \frac{1}{n^{3}}$ & 4 & 61.00 & $8.75 \times 10^{-4}$ & $1.00 \times 10^{-2}$ \\
  & & $\beta = \frac{1}{n^{4}}$ & 4 & 66.05 & $8.76 \times 10^{-4}$ & $1.00 \times 10^{-2}$ \\
  \cline{2-7}
  & \multirow{3}{*}{Algorithm \ref{Alg4}}
  & $\beta = \frac{1}{n^{2}}$ & 5 & 23.93 & $1.94 \times 10^{-2}$ & $4.97 \times 10^{-2}$ \\
  & & $\beta = \frac{1}{n^{3}}$ & 7 & 133.26 & $1.99 \times 10^{-6}$ & $1.00 \times 10^{-2}$ \\
  & & $\beta = \frac{1}{n^{4}}$ & 7 & 140.71 & $2.03 \times 10^{-10}$ & $1.00 \times 10^{-2}$ \\
  \cline{2-7}
  & \multirow{3}{*}{Algorithm \ref{Alg6}}
  & $\beta = \frac{1}{n^{2}}$ & 3 & 4.57 & $1.94 \times 10^{-2}$ & $4.93 \times 10^{-5}$ \\
  & & $\beta = \frac{1}{n^{3}}$ & 4 & 53.90 & $8.75 \times 10^{-4}$ & $1.00 \times 10^{-2}$ \\
  & & $\beta = \frac{1}{n^{4}}$ & 4 & 56.80 & $8.76 \times 10^{-4}$ & $1.00 \times 10^{-2}$ \\
  \bottomrule
  \multirow{20}{*}{0.00001\%}
  & \multicolumn{2}{c}{The power method} & 7 & 20.64 & --- & ---\\
  & \multicolumn{2}{c}{MNM \cite{Guo2018A}} & \multicolumn{4}{c}{Out of memory} \\
  & \multicolumn{2}{c}{Pb-N \cite{Meini2018Perron}} & \multicolumn{4}{c}{Out of memory} \\
  & \multicolumn{2}{c}{ESFPM \cite{Cipolla2019Extrapolation}} & 8 & 13.31 & $2.00 \times 10^{-3}$ & ---\\
  & \multicolumn{2}{c}{EIOM \cite{Cipolla2019Extrapolation}} & 8 & 13.37 & $2.00 \times 10^{-3}$ & ---\\
  \cline{2-7}
  & \multirow{3}{*}{Algorithm \ref{Alg2} \cite{Ding2018Fast}}
  & $\beta = \frac{1}{n^{2}}$ & 4 & 25.88 & $2.00 \times 10^{-2}$ & $1.44 \times 10^{-6}$\\
  & & $\beta = \frac{1}{n^{3}}$ & 6 & 78.01 & $2.00 \times 10^{-2}$ & $9.99 \times 10^{-4}$ \\
  & & $\beta = \frac{1}{n^{4}}$ & 6 & 82.39 & $2.00 \times 10^{-2}$ & $9.99 \times 10^{-4}$\\
  \cline{2-7}
  & \multirow{3}{*}{Algorithm \ref{Alg3}}
  & $\beta = \frac{1}{n^{2}}$ & 3 & 12.03 & $2.00 \times 10^{-3}$ & $5.50 \times 10^{-7}$\\
  & & $\beta = \frac{1}{n^{3}}$ & 5 & 53.21 & $2.61 \times 10^{-6}$ & $9.99 \times 10^{-4}$ \\
  & & $\beta = \frac{1}{n^{4}}$ & 5 & 52.90 & $2.47 \times 10^{-6}$ & $9.99 \times 10^{-4}$\\
  \cline{2-7}
  & \multirow{3}{*}{Algorithm \ref{Alg5}}
  & $\beta = \frac{1}{n^{2}}$ & 3 & 4.47 & $2.00 \times 10^{-3}$ & $5.50 \times 10^{-7}$\\
  & & $\beta = \frac{1}{n^{3}}$ & 3 & 10.50 & $2.23 \times 10^{-4}$ & $9.99 \times 10^{-4}$ \\
  & & $\beta = \frac{1}{n^{4}}$ & 3 & 10.65 & $2.23 \times 10^{-4}$ & $9.99 \times 10^{-4}$\\
  \cline{2-7}
  & \multirow{3}{*}{Algorithm \ref{Alg4}}
  & $\beta = \frac{1}{n^{2}}$ & 3 & 11.62 & $2.00 \times 10^{-3}$ & $5.50 \times 10^{-7}$\\
  & & $\beta = \frac{1}{n^{3}}$ & 5 & 52.94 & $2.00 \times 10^{-7}$ & $9.99 \times 10^{-4}$ \\
  & & $\beta = \frac{1}{n^{4}}$ & 5 & 51.34 & $3.67 \times 10^{-11}$ & $9.99 \times 10^{-4}$\\
  \cline{2-7}
  & \multirow{3}{*}{Algorithm \ref{Alg6}}
  & $\beta = \frac{1}{n^{2}}$ & 3 & 4.27 & $2.00 \times 10^{-3}$ & $5.50 \times 10^{-7}$\\
  & & $\beta = \frac{1}{n^{3}}$ & 3 & 10.55 & $2.23 \times 10^{-4}$ & $9.99 \times 10^{-4}$\\
  & & $\beta = \frac{1}{n^{4}}$ & 3 & 10.66 & $2.23 \times 10^{-4}$ & $9.99 \times 10^{-4}$\\
  \bottomrule
\end{tabular}
\end{table}

\end{document}